\documentclass[11pt,dvipsnames,reqno,twoside a4paper]{amsart}
\usepackage{a4wide}
\usepackage[T1]{fontenc}
\usepackage{dsfont}
\usepackage{yhmath,mathrsfs,amsthm,amsmath,amssymb,amsfonts, enumerate,lipsum, appendix,mathtools, bm,esint}
\usepackage{bbold}
\usepackage[mathscr]{euscript}
\usepackage[normalem]{ulem} 
\usepackage{graphicx}
\usepackage{pgf,tikz}
\usepackage{tkz-euclide}
\usepackage{graphicx}
\usepackage{subcaption}
\usetikzlibrary{shapes.geometric}
\usetikzlibrary{arrows,hobby}
\usepackage{enumitem}
\usepackage[english]{babel}

\allowdisplaybreaks

\usepackage[sort,numbers]{natbib}
\usepackage{orcidlink}
\usepackage{hyperref}
\hypersetup{
	colorlinks=true,
	linkcolor=vd,
	filecolor=blue,      
	urlcolor=dt,
	citecolor=dt,
}
\usepackage[margin=1in, heightrounded]{geometry}
\definecolor{dt}{rgb}{0.73, 0.31, 0.28}
\definecolor{vd}{rgb}{0.0, 0.2, 0.67}
\usepackage{orcidlink}
\newtheorem{theorem}{Theorem}[section]
\newtheorem{lemma}[theorem]{Lemma}
\newtheorem{proposition}[theorem]{Proposition}
\newtheorem{definition}[theorem]{Definition}
\theoremstyle{definition}

\newtheorem{remark}[theorem]{Remark}
\numberwithin{equation}{section}
\usepackage[hyperpageref]{backref}

\newcommand*\rd{\mathbb{R}^d}

\newcommand*\N{\mathcal{N}}
\newcommand{\al} {\alpha}
\newcommand{\ep} {\var}
\newcommand{\pa} {\partial}
\newcommand{\be} {\beta}

\newcommand{\Om} {\Omega}
\newcommand{\la} {\lambda}

\newcommand{\si} {\sigma}

\newcommand{\no} {\nonumber}
\newcommand{\noi} {\noindent}
\newcommand{\var} {\varepsilon}
\newcommand{\ra} {\rightarrow}

\newcommand{\ov} {\overline}

\def\t{\tau}
\def\d{\,{\rm d}}
\def\dx{\,{\rm d}x}
\def\dy{\,{\rm d}y}

\def\C{{\mathcal C}}
\def\D{{\mathcal D}}

\def\M{{\mathcal M}}
\def\n{{\mathcal N}}

\def\W{{\mathcal W}}

\def\R{{\mathbb R}}
\def\N{{\mathbb N}}

\def\({{\Big(}}
\def\){{\Big)}}
\def\cc{{\C_c^\infty}}

\def\dx{\,{\rm d}x}
\def\dX{\,{\rm d}x_m}
\def\dz{\,{\rm d}z}

\def\dxy{\,{\rm d}x{\rm d}y}
\DeclarePairedDelimiter\abs{\lvert}{\rvert}%
\DeclarePairedDelimiter\norm{\lVert}{\rVert}%
\def\wps{{W_0^{s,p}(\Omega )}}

\def\psm{{p^*_s(m)}}

\title[Fractional $p$-Laplace Systems with critical Hardy Nonlinearities]{Fractional $p$-Laplace Systems with critical Hardy Nonlinearities: Existence and Multiplicity}
\author[N. Biswas, P. Das and S. Gupta]{Nirjan Biswas\,$^{1,\dagger}$\orcidlink{0000-0002-3528-8388}, Paramananda Das\,$^2$\,\orcidlink{0009-0009-2821-6675} \and Shilpa Gupta\,$^3$\ \orcidlink{0000-0002-4080-9782}}
\address{\rm $^{1,2}$\,Department of Mathematics, Indian Institute of Science Education and Research (IISER-Pune), Dr. Homi Bhaba Road, Pune-411008, India}
\address{\rm $^3$\,Department of Mathematics and Statistics, Indian Institute of Technology Kanpur, Kanpur-208016, India}

\email[N. Biswas]{nirjan.biswas@acads.iiserpune.ac.in, nirjaniitm@gmail.com} 
\email[P. Das]{pd348225@gmail.com, paramananda.das@students.iiserpune.ac.in}
\email[S. Gupta]{shilpagupta890@gmail.com, shilpa@iitk.ac.in}

\thanks{$^\dagger$Corresponding author}
\begin{document}
\begin{abstract}
Let $\Omega \subset \mathbb{R}^d$ be a bounded open set containing zero, $s \in (0,1)$ and $p \in (1, \infty)$. In this paper, we first deal with the existence, non-existence and some properties of ground-state solutions for the following class of fractional $p$-Laplace systems
\begin{equation*}
\left\{\begin{aligned}
&(-\Delta_p)^s u= \frac{\al}{q} \frac{|u|^{\alpha-2}u|v|^{\beta}}{\abs{x}^m} \;\;\text{in}\;\Omega,\\
&(-\Delta_p)^s v= \frac{\beta}{q} \frac{|v|^{\beta-2}v|u|^{\alpha}}{\abs{x}^m}\;\;\text{in}\;\Omega,\\
&u=v=0\, \mbox{ in }\mathbb{R}^d\setminus \Omega,
\end{aligned}
\right.
\end{equation*}
where $d>sp$, $\alpha + \beta = q$ where $p \leq q \leq p_{s}^{*}(m)$ where $p_{s}^{*}(m) = \frac{p(d-m)}{d-sp}$ with $0 \leq m \leq sp$. Additionally, we establish a concentration-compactness principle related to this homogeneous system of equations. Next, the main objective of this paper is to study the following non-homogenous system of equations
\begin{equation*}
\left\{\begin{aligned}
&(-\Delta_p)^s u = \eta \abs{u}^{r-2}u + \gamma \frac{\alpha}{p_{s}^{*}(m)} \frac{|u|^{\alpha-2}u|v|^{\beta}}{\abs{x}^m} \;\;\text{in}\;\Omega,\\
&(-\Delta_p)^s v = \eta \abs{v}^{r-2}v + \gamma \frac{\beta}{p^{*}_{s}(m)} \frac{|v|^{\beta-2}v|u|^{\alpha}}{\abs{x}^m}\;\;\text{in}\;\Omega,\\
&u=v=0\, \mbox{ in }\mathbb{R}^d\setminus \Omega,
\end{aligned}
\right.
\end{equation*}
where $\eta, \gamma > 0$ are parameters and $p \leq r < p_{s}^{*}(0)$. Depending on the values of $\eta, \gamma$, we obtain the existence of a non semi-trivial solution with the least energy. Further, for $m=0$, we establish that the above problem admits at least $\text{cat}_\Omega(\Omega)$ nontrivial solutions. 
\end{abstract}
\maketitle
\section{Introduction}
Let $\Omega \subset \mathbb{R}^d$ be a bounded open set containing $0$. For $0<s<1<p<\infty$ and $d>sp$, this paper deals with the following critical system of equations driven by the fractional $p$-Laplace operators: 
\begin{equation}\tag{$\mathcal S$}\label{system_eqn}
\left\{\begin{aligned}
&(-\Delta_p)^s u = \eta \abs{u}^{r-2}u + \gamma \frac{\al}{q} \frac{|u|^{\alpha-2}u|v|^{\beta}}{\abs{x}^m} \;\;\text{in}\;\Omega,\\
&(-\Delta_p)^s v = \eta \abs{v}^{r-2}v + \gamma \frac{\beta}{q} \frac{|v|^{\beta-2}v|u|^{\alpha}}{\abs{x}^m}\;\;\text{in}\;\Omega,\\
&u=v=0\, \mbox{ in }\mathbb{R}^d\setminus \Omega,
\end{aligned}
\right.
\end{equation}
where for $0 \le m \le sp$, $p^*_s(m) = \frac{p(d-m)}{d-sp}$ is the fractional Hardy-Sobolev exponent, $\eta, \gamma >0$ are parameters, $p \le r < p^*_s$ (where $p^*_s := p^*_s(0)$) and $\al, \be > 1$ be such that $\al + \be = q$  where $p \le q \le \psm$. The fractional $p$-Laplace operator $(-\Delta _p)^s$ is defined as
\begin{equation*}
   (-\Delta)_{p}^{s}u(x) = 2 \lim_{\var \rightarrow 0^{+}} \int_{\mathbb{R}^{d} \backslash B_{\var}(x)} \frac{|u(x) - u(y)|^{p-2}(u(x)-u(y))}{|x-y|^{d+sp}}\, \dy, \quad \text{for}~x \in \mathbb{R}^{d},
\end{equation*} 
where $B_{\var}(x)$ is the ball of radius $\var$ and centred at $x$. The fractional homogeneous Sobolev space $\D^{s,p}(\rd)$ is defined as the closure of $\C_c^\infty(\R^d)$ with respect to the Gagliardo seminorm \begin{equation*}
    [u]_{s,p}^p := \iint\limits_{\rd \times \rd}\frac{|u(x)-u(y)|^p}{|x-y|^{d+sp}}\,\dx\dy.
\end{equation*} In view of \cite[Theorem 3.1]{BrGoVaJu}, $\D^{s,p}(\R^d)$ has the following characterization:
\begin{equation*}
    \D^{s,p}(\rd):=\left\{u\in L^{p_s^*}(\rd):[u]_{s,p}<\infty\right\},
\end{equation*}
The space $\D^{s,p}(\rd)$ is a reflexive Banach space under the norm $[\cdot]_{s,p}$. Now consider the following closed subspace of $\D^{s,p}(\rd)$:
\begin{equation*}
    W_0^{s,p}(\Omega) :=\{u\in \D^{s,p}(\rd) : u=0 \text{ in }\rd \setminus \Om\},
\end{equation*}
endowed with the norm $[\cdot]_{s,p}$, which is equivalent to the fractional Sobolev norm $\|\cdot\|_{L^p(\Om)}+[\cdot]_{s,p}$. For further details on fractional Sobolev spaces and associated embedding results, we refer to \cite{SaSu, DiPaVa} and the references therein. We consider $\mathcal{W} = W_0^{s,p}(\Omega) \times W_0^{s,p}(\Omega)$ as the solution space for \eqref{system_eqn}, which is endowed with the following norm
\begin{align*}
    \norm{(u,v)}_{\W} := \left( [u]_{s,p}^p + [v]_{s,p}^p \right)^{\frac{1}{p}}. 
\end{align*}
The goal of this paper is twofold. First, we study the following homogeneous system of equations:
\begin{equation}\tag{$\mathcal S_H$}\label{system_eqn_homogenous}
\left\{\begin{aligned}
&(-\Delta_p)^s u= \frac{\al}{q} \frac{|u|^{\alpha-2}u|v|^{\beta}}{\abs{x}^m} \;\;\text{in}\;\Omega,\\
&(-\Delta_p)^s v= \frac{\beta}{q} \frac{|v|^{\beta-2}v|u|^{\alpha}}{\abs{x}^m}\;\;\text{in}\;\Omega,\\
&u=v=0\, \mbox{ in }\mathbb{R}^d\setminus \Omega.
\end{aligned}
\right.
\end{equation}
We consider the following quantities:
\begin{align*}
    &S_{\al+\be} := \inf_{u \in \wps \setminus \{ 0\}} \frac{[u]_{s,p}^p}{\left( \int_{\Omega} \frac{\abs{u(x)}^{q}}{\abs{x}^m} \dx \right)^{\frac{p}{q}}} \text{ and }
    S_{\al, \be} := \inf_{u,v \in W_0^{s,p}(\Om) \setminus \{0\}} \frac{\norm{(u,v)}^p_{\W}}{\left( \int_{\Omega} \frac{\abs{u}^{\alpha}\abs{v}^{\beta}}{\abs{x}^m} \dx \right)^{\frac{p}{q}}}.
\end{align*}
Observe that, if $u$ minimizes $S_{\al+\be}$, then so is $\abs{u}$, since $[\abs{u}]_{s,p} \le [u]_{s,p}$. So, without loss of generality, we assume that a minimizer of $S_{\al+\be}$ is nonnegative. A similar fact holds for a minimizer of $S_{\al, \be}$. If $(u, v) \in \mathcal{W}$ minimizes $S_{\al, \be}$, then it satisfies 
\begin{align}\label{weak2}
    &\iint\limits_{\rd\times\rd}\frac{|u(x)-u(y)|^{p-2}(u(x)-u(y))(\phi(x)-\phi(y))}{|x-y|^{d+sp}} \dx \dy \no \\ 
    &+\iint\limits_{\rd\times\rd}\frac{|v(x)-v(y)|^{p-2}(v(x)-v(y))(\psi(x)-\psi(y))}{|x-y|^{d+sp}} \dx \dy \no \\ 
    & =\frac{\la}{q} \left(\al \int_{\Omega} \frac{\abs{u}^{\al-2}u\abs{v}^{\beta}}{\abs{x}^m} \phi \dx + \be \int_{\Omega} \frac{\abs{v}^{\be-2}v\abs{u}^{\al}}{\abs{x}^m} \psi \dx \right),
\end{align}
for every $(\phi, \psi) \in \mathcal{W}$, where
\begin{align*}
    \la = \frac{\norm{(u,v)}^p_{\W}}{\int_{\Omega} \frac{\abs{u}^{\alpha} \abs{v}^{\beta}}{\abs{x}^m} \dx} =  S_{\al, \be} \left( \int_{\Omega} \frac{\abs{u}^{\alpha}\abs{v}^{\beta}}{\abs{x}^m} \dx \right)^{\frac{p}{q}-1}.
\end{align*} 
\begin{remark}\label{remark-1}
If $(u_0,v_0)$ minimizes $S_{\al, \be}$, then the following normalized pair 
\begin{align*}
  (u_1,v_1) :=  \left( S_{\al, \be}^{\frac{1}{q-p}} u_0 \left(\int_{\Omega} \frac{u_0^{\alpha} v_0^{\beta}}{\abs{x}^m} \dx \right)^{-\frac{1}{q}}, S_{\al, \be}^{\frac{1}{q-p}} v_0\left(\int_{\Omega} \frac{u_0^{\alpha} v_0^{\beta}}{\abs{x}^m} \dx \right)^{-\frac{1}{q}} \right)
\end{align*}
also minimizes $S_{\al, \be}$. Applying the Lagrange multiplier and  using \eqref{weak2} it follows that $(u_1,v_1)$ weakly solves \eqref{system_eqn_homogenous}. Moreover, taking $(u_1, 0)$ and $(0,v_1)$ as test functions in $\mathcal{W}$, we see the following equations hold weakly: 
\begin{align*}
    (-\Delta_p)^s u_1 \ge 0, \; (-\Delta_p)^s v_1 \ge 0, \text{ in } \Omega, \; u_1, v_1 = 0 \text{ in } \rd \setminus \Omega. 
\end{align*}
Then, the strong maximum principle of $(-\Delta _p)^s$ (see \cite[Theorem A.1]{BrFr}), infers either $u_1 \equiv 0$ or $u_1>0$ a.e. in $\Omega$. Similarly, either $v_1 \equiv 0$ or $v_1>0$ a.e. in $\Omega$. 
\end{remark}

Now we define the ground state solution of \eqref{system_eqn_homogenous} up to the normalization as mentioned in Remark \ref{remark-1}. 

\begin{definition}
A pair $(u_0, v_0)$ is called a `ground state solution' to \eqref{system_eqn_homogenous}, if it minimizes $S_{\al, \be}$.
\end{definition}

\begin{proposition}\label{existence&uniqueness}
 Let $\Omega$ be a bounded open set containing zero, $0\le m \le sp$, $p \le q \le p^*_s(m)$, and $\al, \be > 1$ be such that $\al + \be = q$. Then the following hold: 
 \begin{enumerate}
     \item[\rm{(i)}]  If $q<\psm$, then \eqref{system_eqn_homogenous} admits a ground state solution. 
     \item[\rm{(ii)}] If $q = p^*_s$, then \eqref{system_eqn_homogenous} does not admit any ground state solution.
 \end{enumerate}
\end{proposition}

For the existence of a ground state solution, we show that the following map 
\begin{align*}
    G(u,v) :=  \int_{\Omega} \frac{\abs{u(x)}^{\al} \abs{v(x)}^{\be}}{\abs{x}^m} \dx, \; \forall \, (u,v) \in \W,
\end{align*}
is weakly continuous on $\W$. However, when $q= p^*_s(m)$, the map $G$ fails to be weakly continuous due to the lack of compactness in the embedding $\W \hookrightarrow L^{\psm}(\Omega, \abs{x}^{-m}) \times L^{\psm}(\Omega, \abs{x}^{-m})$. This necessitates a deeper understanding of the behaviour of weakly convergent sequences within the space of measures. In this paper, we address this by proving a concentration compactness principle associated with \eqref{system_eqn_homogenous}, capturing the limiting behaviour of such sequences (see Proposition \ref{CCP}).  

\medskip

Next, for $q= \psm$, we study the existence of nontrivial solutions for \eqref{system_eqn}, 
which is the system counterpart of the following scalar equation
\begin{equation}
\left\{\begin{aligned}\label{scaler_equation}
&(-\Delta_p)^s u = \eta \abs{u}^{r-2}u + \gamma \frac{\abs{u}^{p^*_s(m)-2}u}{\abs{x}^m} \;\;\text{in}\;\Omega,\\
&u=v=0\, \mbox{ in }\mathbb{R}^d\setminus \Omega.
\end{aligned}
\right.
\end{equation}
In \cite{chen20183065}, Chen et al. proved that \eqref{scaler_equation} admits a positive solution with the least energy, which depends on the values of $\eta, \gamma$. Additionally, they provided the existence of a sign-changing solution for \eqref{scaler_equation}. In the literature, various authors have studied the existence of solutions to elliptic systems with fractional $p$-Laplacian. In \cite{ChSq}, Chen and Squassina considered the following concave-convex system of equations on a bounded open set $\Omega$:
\begin{equation}\label{Chen-Squassina}
\left\{\begin{aligned}
&(-\Delta_p)^s u = \eta_1 \abs{u}^{q-2}u + \frac{\al}{p_{s}^{*}} |u|^{\alpha-2}u|v|^{\beta} \;\;\text{in}\;\Omega, \\
&(-\Delta_p)^s v = \eta_2 \abs{v}^{q-2}v + \frac{\beta}{p^{*}_{s}} |v|^{\beta-2}v|u|^{\alpha}\;\;\text{in}\;\Omega, \\
&u=v=0\;\;\text{in}\; \rd \setminus \Omega,
\end{aligned}
\right.
\end{equation}
where $\eta_1, \eta_2>0$, $q \in (1,p)$ and $\al+\be= p^*_s$. Employing the Nehari manifold technique, it is shown that \eqref{Chen-Squassina} admits at least two non semi-trivial solutions provided $\eta_1, \eta_2$ have certain upper bounds and  $q, N$ adhere to specific constraints depending on $p,s$ (see \cite[Theorem 1.2]{ChSq}). In the critical case $r=p^*_s$ and for $\Omega=\R^d$, Shen in \cite{Shen2021} proved the existence of a solution for \eqref{system_eqn} with general nonlinearities, and Bhakta et al. in \cite[Theorem 1.5]{MoKaFi} obtained the existence of a positive solution for \eqref{system_eqn} with the least energy, under the assumption that $m=0$ and under certain restrictions on $d,p,s, \al, \be$. For a similar study on nonlinear systems with multiple critical nonlinearities, we refer the reader to \cite{GiPaSr, assunccao2024, Hsu} and references therein.  

\medskip
We say a pair $(u,v) \in \mathcal{W}$ is a weak solution of \eqref{system_eqn}, if it satisfies the following identity for every $(\phi, \psi) \in \mathcal{W}$:
\begin{align*}
    &\iint\limits_{\rd\times\rd}\frac{|u(x)-u(y)|^{p-2}(u(x)-u(y))(\phi(x)-\phi(y))}{|x-y|^{d+sp}} \dx \dy \no \\ 
   &+\iint\limits_{\rd\times\rd}\frac{|v(x)-v(y)|^{p-2}(v(x)-v(y))(\psi(x)-\psi(y))}{|x-y|^{d+sp}} \dx \dy \no \\ 
   &=\gamma \frac{\al}{p^*_s(m)} \int_{\Omega} \frac{\abs{u}^{\al-2}u\abs{v}^{\beta}}{\abs{x}^m} \phi \dx + \gamma \frac{\be}{p^*_s(m)} \int_{\Omega} \frac{\abs{v}^{\be-2}v\abs{u}^{\al}}{\abs{x}^m} \psi \dx \no \\
   &+\eta \int_{\Omega} \abs{u}^{r-2}u \phi + \eta \int_{\Omega} \abs{v}^{r-2}v \psi.
\end{align*}
We note that system \eqref{system_eqn} is variational, and the underlying functional is
\begin{equation}\label{energy-main}
\begin{split}
I(u,v):= \frac{1}{p}\|(u,v)\|^p_{\mathcal{W}}
&-\frac{\gamma}{p^*_s(m)}\int_{\Omega} \frac{|u|^{\alpha}|v|^{\beta}}{\abs{x}^m} \dx - \frac{\eta}{r} \int_{\Omega} \left( \abs{u}^r + \abs{v}^r \right) \dx, \; \forall \, (u,v) \in \W, 
 \end{split}
\end{equation}
which is of class $\C^1(\mathcal{W})$. Moreover, if $(u,v)$ is a weak solution of~\eqref{system_eqn}, then $(u,v)$ is a critical point of $I$ and vice versa.

\medskip 
In this paper, we extend the result of \cite[Theorem 1.2]{chen20183065} in the context of the system. Specifically, we prove the existence of a nontrivial solution with the least energy for \eqref{system_eqn}. Moreover, we prove that the nontrivial solution is indeed non semi-trivial. To prove the non semi-triviality, we note that the approach used in \cite{ChSq}, which relies on the presence of a concave nonlinearity, does not apply to our setting. Instead, we employ an alternative approach based on the properties of the Nehari manifold. 

\medskip
We consider the following quantity
\begin{align}\label{c0}
c_0:=\left(\int_{\Omega}|x|^{\frac{pm}{p_{s}^{*}(m)-p}}\right)^{\frac{p_{s}^{*}(m)-p}{p_{s}^{*}(m)}}.    
\end{align}
Recall the first Dirichlet eigenvalue for $(-\Delta _p)^s$ (see \cite{FrPa}):
\begin{equation}\label{lm_1}
\lambda_1 := \inf_{u\in \wps\setminus\{0\}}\frac{[u]_{s,p}^p}{\|u\|^{p}_{L^{p}(\Omega)}}>0.\end{equation} 
Now, we are in a position to state our result.  
\begin{theorem}\label{existence-main}
Let $\Omega$ be a bounded open set containing zero, $0 \leq m \le sp$, $p \le r < p^*_s$, and $\max \left\{r, p_{s}^{*}(m)\right\}>p$. Let $\al,\be$ be such that $\al + \be = p^*_s(m)$.  Choose $\eta, \gamma>0$ such that
\begin{align*}
    &\eta < \min\left\{\lambda_1, \frac{S_{\al+\be}}{c_{0}}\right\}, \text{ when } m<sp, r=p, \text{ and } \gamma < S_{\al, \be}, \text{ when } m=sp, r>p.
\end{align*}
Then the following hold:
\begin{enumerate}
    \item[\rm{(i)}] \eqref{system_eqn} admits a nontrivial weak solution $(\tilde{u}, \tilde{v})$ with the least energy provided
\begin{align*}
\begin{cases}r \geq p, & \text { if } d \geq p^{2} s; \\ r>p^{*}-p^{\prime}, & \text { if } d<p^{2} s.\end{cases}
\end{align*}
\item[\rm{(ii)}] Further, there exists $0<\eta^* < \min\{\lambda_1, \frac{S_{\al+\be}}{c_{0}}\}$ such that for every $\eta \in (0, \eta^*)$, $(\tilde{u}, \tilde{v})$ is non semi-trivial, i.e., $(\tilde{u}, \tilde{v}) \neq (\tilde{u},0)$ or  $(\tilde{u}, \tilde{v}) \neq (0,\tilde{v})$ a.e. in $\Omega$.
\end{enumerate}
\end{theorem}
We first demonstrate the existence of a nontrivial solution using the Mountain Pass theorem. Subsequently, we introduce the Nehari manifold to show that the solution obtained via the Mountain Pass Theorem attains the least energy. Adopting similar variational arguments for $I_+$ (defined in \eqref{I+}) as used in the preceding theorem, we get the existence of a positive solution for \eqref{system_eqn} (see Remark \ref{pos}).

\medskip
Next, we study the multiplicity of non-trivial solutions for \eqref{system_eqn} employing Ljusternik-Schnirelmann category theory (see \cite{palais}, \cite[Chapter 5]{willem} or \cite[Chapter 9]{AmMa}). Recall that the Ljusternik-Schnirelmann category $\text{cat}_X(A)$ of a closed subset $A$ of a topological space $X$ is the least integer $n$ such that there exists a covering of $A$ by $n$ closed sets contractible in $X$. In the classical paper \cite{BeCe}, Benci and Cerami proved the existence of $\text{cat}_\Om(\Om)$ distinct critical points for a Brezis-Nirenberg type subcritical problem associated with the Laplace operator. Later, this theory has been used in various critical problems. To name a few, Alves and Ding \cite{AlDi} used the LS category theory to study the multiplicity of positive solutions of the following local problem
$$-\Delta_pu=\mu u^{q-1}+u^{p^*-1}\text{ in } \Om,\quad  2\leq p\leq q<p^*=\frac{dp}{d-p}.$$
In \cite{FiMoSe}, Figueiredo et al. used the LS category theory for the following fractional problem 
$$(-\Delta)^su=\mu |u|^{q-2}u+|u|^{2^*_s-2}u \text{ in } \Om,\,u=0 \text{ in } \rd \setminus \Omega,\quad  2\leq q<2_s^*=\frac{2d}{d-2s}.$$
Recently, da Silva et al. \cite{DFV} studied the multiplicity of solutions of a mixed local-nonlocal quasilinear problem:
$$-\Delta_p u+(-\Delta_p)^s u=\la|u|^{q-2}u+|u|^{p^*-2}u \text{ in }\Om,\; u=0\text{ in }\rd \setminus \Omega, \quad 2\leq p<q<p^*=\frac{dp}{d-p},$$ by using the LS category theory.
A similar multiplicity result is recently proved in \cite{EcMoMoRa} for a fractional $p$-Laplacian system with concave nonlinearities. For $m=0$, the following theorem states category many nontrivial solutions for \eqref{system_eqn}.

\begin{theorem}\label{cat_theorem}
    Let $\Omega$ be a bounded open set containing zero and $m=0$. Let $2\leq p\leq r<p_s^*$, and $\al,\be\geq2$ be such that $\al + \be = p^*_s$. Let $\gamma>0$, $r$ be as in Theorem \ref{existence-main} and $c_0$ be as in \eqref{c0}. Then there exists $\eta_*>0$ satisfying 
    $$\eta_*<\min\left\{\frac{\la_1}{2},\frac{S_{\al+\be}}{c_0}\right\}, \text{ when } r=p,$$
    such that for every $\eta\in(0,\eta_*)$, \eqref{system_eqn} has at least $\text{cat}_\Om(\Om)$ nontrivial weak solutions. 
\end{theorem}
\begin{remark}
    If $\Om$ is contractible in itself, then the above theorem doesn't give any extra information about the multiplicity of solutions. However if $\Om$ is an annular region then $\text{cat}_\Om(\Om)=2$. In that case, the above theorem yields at least two distinct nontrivial weak solutions of \eqref{system_eqn}. In short, the above theorem says that if $\Om$ is topologically rich, then we get more information about the multiplicity of weak solutions of \eqref{system_eqn}.
    \end{remark}
The rest of the paper is structured as follows: Section \ref{exst_unq_grnd} is dedicated to proving Proposition \ref{existence&uniqueness}. In this section, we also discuss various properties of the minimizer for $S_{\al,\be}$. Section \ref{conc_comp} contains the study of the concentration compactness principle associated with \eqref{system_eqn_homogenous}. In Section \ref{exst_nontriv}, we verify the Palais-Smale condition, analyze the mountain pass geometry, and leverage these results to prove Theorem \ref{existence-main}. Finally, in the last section, we establish the multiplicity results in Theorem \ref{cat_theorem}.

\medskip
\noi \textbf{Notation:} We use the following notation and convention:
\begin{enumerate}
    \item[(i)] $C$ denotes a generic positive constant.
    \item[(ii)] For $a \in (1, \infty)$, $a' := \frac{a}{a-1}$ is the conjugate exponent of $a$. 
    \item[(iii)] $\omega_d$ represents the measure of a unit ball in $\rd$.
    \item[(iv)] $\W(U):=W_0^{s,p}(U)\times W_0^{s,p}(U)$, for any open set $U\subset\R^d$ and we reserve the notation $\W$ for $\W(\Om)$. Also, $\mathcal{W}(\rd) := \mathcal{D}^{s,p}(\rd) \times \mathcal{D}^{s,p}(\rd)$.
    \item[(v)]  $\dX := \dx \abs{x}^{-m}$.
    \item[(vi)] $B_{r}(x)$ denotes a ball of radius $r$ in $\rd$ with centre at $x$.
\end{enumerate}

\section{Existence and Non-existence of ground state solutions}\label{exst_unq_grnd}
Recall the Hardy-Sobolev inequality (see \cite[Theorem 1.1]{FrSe}):
\begin{align}\label{HS}
    C(d,s,p) \int_{\Omega} \frac{\abs{u(x)}^p}{\abs{x}^{sp}} \dx \le \iint\limits_{\rd \times \rd}\frac{|u(x)-u(y)|^p}{|x-y|^{d+sp}}\,\dx\dy, \; \forall \, u \in \wps.
\end{align}
Next, we recall the following fractional Hardy Sobolev inequality is the interpolation between \eqref{HS} and the critical Sobolev inequality (for $m=0$):
\begin{equation}\label{HS1}
    C(d,s,p,m) \left( \int_{\Omega} \frac{\abs{u(x)}^{p^*_s(m)}}{\abs{x}^m} \dx \right)^{\frac{p}{p^*_s(m)}} \le  \iint\limits_{\rd \times \rd}\frac{|u(x)-u(y)|^p}{|x-y|^{d+sp}}\,\dx\dy, \; \forall \, u \in W_0^{s,p}(\Omega).  
\end{equation}
For $q \in [p, \psm]$, using \eqref{HS1} in \cite[Lemma 2.3]{chen20183065} the authors verified the following inequality: 
\begin{align}\label{HS1.5}
    C(d,s,p,m) \left( \int_{\Omega} \frac{\abs{u(x)}^q}{\abs{x}^m} \dx \right)^{\frac{p}{q}} \le  \iint\limits_{\rd \times \rd}\frac{|u(x)-u(y)|^p}{|x-y|^{d+sp}}\,\dx\dy, \; \forall \, u \in W_0^{s,p}(\Omega). 
\end{align}
For $\al+\be = q$, using the H\"{o}lder's inequality with the conjugate pair $(\frac{q}{\al}, \frac{q}{\be})$ and \eqref{HS1.5} we get the following estimate for all $(u,v) \in \W$:
\begin{align}\label{HS2}
    &\int_{\Omega} \frac{\abs{u(x)}^{\al} \abs{v(x)}^{\be}}{\abs{x}^m} \dx = \int_{\Omega} \frac{\abs{u(x)}^{\al}}{\abs{x}^{\frac{m \al}{q}}}  \frac{\abs{v(x)}^{\be}}{\abs{x}^{\frac{m \be}{q}}} \dx \no \\
    %&\le \left( \int_{\Omega} \frac{\abs{u(x)}^{p^*_s(m)}}{\abs{x}^m} \dx \right)^{\frac{\al}{p^*_s(m)}} \left( \int_{\Omega} \frac{\abs{v(x)}^{p^*_s(m)}}{\abs{x}^m} \dx \right)^{\frac{\be}{p^*_s(m)}} \no \\
    &\le C(d,s,p,m) \left( \; \iint\limits_{\rd \times \rd}\frac{|u(x)-u(y)|^p}{|x-y|^{d+sp}}\,\dx\dy \right)^{\frac{\al}{p}} \left( \; \iint\limits_{\rd \times \rd}\frac{|v(x)-v(y)|^p}{|x-y|^{d+sp}}\,\dx\dy \right)^{\frac{\be}{p}} \no \\
    &\le C(d,s,p,m) \norm{(u,v)}_{\mathcal{W}}^{q}.
\end{align} 
Clearly, \eqref{HS2} yields
\begin{align}\label{HS3}
    \left( \int_{\Omega} \frac{\abs{u(x)}^{\al} \abs{v(x)}^{\be}}{\abs{x}^m} \dx \right)^{\frac{p}{q}} \le C(d,s,p,m) \norm{(u,v)}_{\mathcal{W}}^p, \; \forall \, (u,v) \in \mathcal{W}.
\end{align}
From the definition, observe that $S_{\al+\be}$ and $S_{\al,\be}$ are the best constants of \eqref{HS1.5} and \eqref{HS3} respectively. Consider the following map
\begin{align*}
    G(u,v) :=  \int_{\Omega} \frac{\abs{u(x)}^{\al} \abs{v(x)}^{\be}}{\abs{x}^m} \dx, \; \forall \, (u,v) \in \W.
\end{align*} 
Observe that $G\in \C^1(\mathcal{W})$ and for all $(u,v) \in \W$, 
\begin{align*}
    &G'(u,v)(\phi, \psi) = \al \int_{\Omega} \frac{\abs{u}^{\al-2} u \abs{v}^{\be}}{\abs{x}^m} \phi \dx +  \be \int_{\Omega} \frac{\abs{v}^{\be-2} v \abs{u}^{\al}}{\abs{x}^m}  \psi \dx, \; \forall \, (\phi, \psi) \in \W.
\end{align*}
Next, we require the following lemma. The proof follows the same arguments as in \cite[Theorem 2]{brezis1983relation} (with $\dx$ replacing by $\frac{\dx}{\abs{x}^{m}}$). 

\begin{lemma}\label{convergence1}
    Let $j : \mathbb{C} \ra \mathbb{C}$ be a continuous function with $j(0) = 0$. Given $\ep>0$, assume that there exists two continuous functions $\phi_{\var}$ and $\psi_{\var}$ such that the following inequality holds: 
\begin{align*}
    \abs{j(a+b) - j(a)} \le \ep \phi_{\ep}(a) + \psi_{\ep}(b), \; \forall \, a,b \in \mathbb{C}. 
\end{align*}
Let $m \ge 0$. Further, let $f_n: \rd \ra \mathbb{C}$ with $f_n = f+ g_n$ be a sequence of measurable functions such that 
\begin{enumerate}
    \item[\rm{(i)}] $g_n \ra 0$ a.e. in $\rd$,
    \item[\rm{(ii)}] $j(f) \in L^1(\abs{x}^{-m}, \rd)$,
    \item[\rm{(iii)}] $\int_{\rd} \phi_{\ep}(g_n(x)) \frac{\dx}{\abs{x}^m} \le C < \infty$, for some constant $C$ independent of $\ep$ and $d$. 
    \item[\rm{(iv)}] $\int_{\rd} \psi_{\ep}(f(x)) \frac{\dx}{\abs{x}^m} < \infty$ for all $\var> 0$.
\end{enumerate}
Then 
\begin{align}\label{converge-1}
   \int_{\rd} \left| j(f_n) - j(f) - j(g_n) \right| \frac{\dx}{\abs{x}^m} = o_n(1).
\end{align}
\end{lemma}

The above lemma yields the following convergence. 

\begin{lemma}\label{convergence2}
Let  $0\le m \le sp$, $p \le q < p^*_s(m)$ and $\al, \be > 1$ be such that $\al + \be = q$. Let $\{ (u_n, v_n) \}$ weakly converges to $ {(u,v)}$ in $\mathcal{W}$. Then 
\begin{align*}
   \int_{\Omega} \left( \abs{u_n}^{\alpha} \abs{v_n}^{\beta} - \abs{u}^{\alpha} \abs{v}^{\beta} - \abs{u_n - u}^{\alpha} \abs{v_n-v}^{\beta} \right) \frac{\dx}{\abs{x}^m} = o_n(1).
\end{align*}
\end{lemma}

\begin{proof}
    Consider $j : \R^2 \ra \R$ such that $j(x,y) = \abs{x}^{\al} \abs{y}^{\be}$. Then, in view of the inequality given in \cite[Lemma 3.2]{MoSoMiPa}, $j$ satisfies the hypothesis used in Lemma \ref{convergence1} with $\phi_{\var}(x,y) = \psi_{\var}(x,y)= \abs{x}^{\al+ \be} + \abs{y}^{\al + \be}$.  Now we consider 
    \begin{align*}
        f_n := (u_n, v_n), f := (u,v), \text{ and } g := (u_n - u, v_n -v). 
    \end{align*}
    Using the boundedness of $\{ u_n \}, \{ v_n \}$ in $\wps$ and using the inequalities \eqref{HS1.5} and \eqref{HS3}, we can verify that (i)-(iv) hold in Lemma \ref{convergence1}. Therefore, we get \eqref{converge-1} by applying Lemma \ref{convergence1}.   
\end{proof}

For $q \neq \psm$, we first investigate the existence of ground state solutions for \eqref{system_eqn_homogenous} by examining the existence of minimizers for $S_{\alpha, \beta}$.

\begin{proposition}\label{compact&existence}
Let $0 \leq m \le sp$, $p \le q < p^*_s(m)$ and $\al, \be > 1$ be such that $\al + \be =q$. Then the following hold:
   \begin{enumerate}
       \item[\rm{(i)}] The map $G$ is weakly continuous on $\W$.
       \item[\rm{(ii)}] Let $\M = \left\{ (u,v) \in \W:  G(u,v) = 1 \right\}$. Then 
      $S_{\al, \be} = \inf \left\{ \norm{(u,v)}^p_{\W}: (u,v) \in \M  \right\}$ admits a positive minimizer in $\M$.
        % \item[\rm{(iii)}] Let $\mathcal{N} = \left\{ u \in \wps, H(u) = 1 \right\}.$ Then $S_{\al+\be} := \inf \left\{[u]_{s,p}^p: u \in \mathcal{N}  \right\}$
        %  admits a positive minimizer in $\mathcal{N}$.
   \end{enumerate}
\end{proposition}
\begin{proof}
(i) Let $(u_n,v_n) \rightharpoonup (u,v)$ in $\W$. Using Lemma \ref{convergence2}, 
\begin{equation}\label{par_01_01}
    \int_{\Omega}  \frac{\abs{\abs{u_n}^{\al}\abs{v_n}^{\be} - \abs{u}^{\al}\abs{v}^{\be}}}{\abs{x}^m}  \dx = \int_{\Om}\frac{\abs{u_n-u}^{\al}\abs{v_n-v}^{\be}}{\abs{x}^m} \dx+o_n(1).
\end{equation} 
Now, using the H\"{o}lder's inequality with the conjugate pair $(\frac{q}{\al},\frac{q}{\be})$, 
\begin{equation}\label{par_01_02}
\begin{aligned}
    \int_{\Om}\frac{\abs{u_n-u}^{\al}\abs{v_n-v}^{\be}}{\abs{x}^m}  \dx&=\int_{\Om}\frac{\abs{u_n-u}^{\al}\abs{v_n-v}^{\be}}{\abs{x}^{\frac{\al m}{q}}\abs{x}^{\frac{\be m}{q}}}  \dx\\
    &\leq \left(\int_{\Om}\frac{\abs{u_n-u}^{q}}{|x|^{m}}\dx\right)^{\frac{\al}{q}}\left(\int_{\Om}\frac{\abs{v_n-v}^{q}}{|x|^{m}}\dx\right)^{\frac{\be}{q}}.
\end{aligned}
\end{equation}
Now proceeding as in \cite[Lemma 2.3]{chen20183065}, we estimate  
\begin{align*}
    \int_{\Om}\frac{\abs{u_n-u}^{q}}{|x|^{m}}\dx & = \int_{\Omega} \frac{\abs{u_n-u}^{\frac{m}{s}}}{\abs{x}^{m}}\abs{u_n - u}^{q- \frac{m}{s}} \dx \\
    & \le \left( \int_{\Omega} \frac{\abs{u_n -u}^p}{\abs{x}^{sp}} \dx \right)^{\frac{m}{sp}} \left( \int_{\Omega} \abs{u_n - u}^{\sigma} \dx \right)^{\frac{sp-m}{sp}} \le C \left( \int_{\Omega} \abs{u_n - u}^{\sigma} \dx \right)^{\frac{sp-m}{sp}},
\end{align*}
where $\si=(q-\frac ms)\frac{ps}{ps-m}$ and the last inequality follows using \eqref{HS} and the boundedness of $\{u_n\}$ in $\wps$.
Further, using $q < \psm$ we note that $\sigma < p^*_s$ and hence using the compact embedding of $\wps \hookrightarrow L^{\sigma}(\Omega)$, we get
\begin{align*}
    \int_{\Omega} \abs{u_n - u}^{\sigma} \dx \ra 0, \text{ as } n \ra \infty.
\end{align*}
Similarly, we also get $\int_{\Om}\frac{\abs{v_n-v}^{q}}{|x|^{m}}\dx \ra 0$ as $n \ra \infty$. Therefore, taking $n \ra \infty$ in \eqref{par_01_01} and \eqref{par_01_02}, we conclude that $G$ is weakly continuous.

\noi (ii) Let $\{ (u_n,v_n) \}$ be a minimizing sequence in $\mathcal{M}$ such that $\norm{(u_n , v_n)}^p_{\W}\ra  S_{\al, \be}.$ By the reflexivity, up to a subsequence $u_n \rightharpoonup u_0$ and $v_n \rightharpoonup v_0$ in $\wps$. Since $G$ is sequentially compact, 
    \begin{align*}
        \int_{\Omega} \frac{\abs{u_0}^{\alpha}\abs{v_0}^{\beta}}{\abs{x}^m} \dx = \lim_{n \ra \infty} \int_{\Omega} \frac{\abs{u_n}^{\alpha}\abs{v_n}^{\beta}}{\abs{x}^m} \dx = 1,
    \end{align*}
    which implies $(u_0, v_0) \in \M$. Further, using the weak lower semicontinuity of $\norm{(\cdot, \cdot)}_{\W}$ and the definition of $S_{\al,\be}$, 
    \begin{align*}
        S_{\al, \be} = \lim_{n \ra \infty} \norm{(u_n , v_n)}^p_{\W} \ge \norm{(u_0 , v_0)}^p_{\W} \ge S_{\al, \be}.
    \end{align*}
    Therefore, $(u_0, v_0)$ is a minimizer for $S_{\al, \be}$. Moreover, using Remark \ref{remark-1}, $(u_0, v_0)$ is positive.
% \noi (iv) Applying the direct variational arguments (as in (iii)), we get the existence of a minimizer $w_0$ with $w_0 \ge 0$ a.e. in $\Omega$. Moreover, for every $\phi \ge 0$ in $\wps$, 
 %   \begin{equation*}
 %       \begin{split}
 %           \iint\limits_{\rd\times\rd}\frac{|w_0(x)-w_0(y)|^{p-2}(w_0(x)-w_0(y))(\phi(x)-\phi(y))}{|x-y|^{d+sp}} \dx \dy \\ = \left( \frac{[w_0]_{s,p}^p}{\int_{\Omega} {w_0}^{p^*_s(m)} \dX} \right) \int_{\Omega} \frac{w_0^{\psm-1}}{\abs{x}^m} \phi \dx \ge 0.
 %       \end{split}
 %   \end{equation*}
 % Therefore, the strong maximum principle for the fractional $p$-Laplace operator \cite[Theorem A.1]{BrFr} yields $w_0> 0$ a.e. in $\Omega$.    
\end{proof}

In the following proposition, inspired by \cite[Lemma 5.1]{FaMiPeZh}, we provide a relation between $S_{\al+\be}$ and $S_{\al, \be}$. 
\begin{proposition}\label{relation1}
    Let $0 \leq m \le sp$, $p \le q \le p^*_s(m)$ and $\al, \be > 1$ be such that $\al + \be =q$. Then
\begin{align*}
    S_{\al, \be} = \left( \left( \frac{\al}{\be} \right)^{\frac{\be}{\al + \be}} + \left( \frac{\al}{\be} \right)^{-\frac{\al}{\al+\be}} \right) S_{\al+\be}.
\end{align*}
Let $w_0$  be a minimizer of $S_{\al+ \be}$. Then $(Bw_0,Cw_0)$ achieves $S_{\al, \be}$ for all constants $B,C>0$ satisfying $\frac{B}{C} = ( \al \be^{-1})^{\frac{1}{p}}$.
\end{proposition}

The following proposition shows specific connections between $(u_0, v_0)$ and $(B w_0, C w_0)$. 

\begin{proposition}
  Let $0 \leq m \le sp$, $p \le q \le p^*_s(m)$ and $\al, \be>1$ be such that $\al + \be =q$. Let $w_0$  be a minimizer of $S_{\al+ \be}$ and $(u_0,v_0)$ be a minimizer of $S_{\al,\be}$. Then for every $B,C>0$ with $\frac{B}{C} = ( \al \be^{-1})^{\frac{1}{p}}$, \begin{align}\label{re2}
       \int_{\Omega} u_0(x)^{\al} v_0(x)^{\be} \dX = B^{\al} C^{\be} \int_{\Omega} w_0(x)^{q} \dX, 
    \end{align} 
\end{proposition}

\begin{proof}
  Consider the following quantity 
    \begin{align*}
        F := \inf_{(u,v) \in \W} \max_{t>0} E(tu, tv), \text{ with } E(u,v) = \frac{X}{p} - \frac{Y}{q},
    \end{align*}
    where $X= \norm{(u,v)}^p_{\W}$ and $Y = \int_{\Omega} \abs{u}^{\al} \abs{v}^{\be} \dX$. Observe that 
    \begin{align*}
     \text{for }  \tilde{t} = \left( \frac{X}{Y} \right)^{\frac{1}{q -p}}, \; E(\tilde{t}u,\tilde{t}v) =\max_{t>0} E(tu,tv).
    \end{align*}
     Hence
    \begin{align*}
        F = \inf_{(u,v) \in \W} \left( \frac{(\tilde{t})^p}{p} X - \frac{(\tilde{t})^{q}}{q} Y \right) &= \inf_{(u,v) \in \W} \left( \frac{1}{p} - \frac{1}{q} \right) \frac{X^{\frac{q}{q -p}}}{Y^{\frac{p}{q -p}}} \\
        & = \left( \frac{1}{p} - \frac{1}{q} \right) \left( \inf_{(u,v) \in \W} \frac{X}{Y^{\frac{p}{q}}} \right)^{\frac{q}{q - p}}.
    \end{align*}
    Since $(u_0, v_0)$ weakly satisfies \eqref{system_eqn_homogenous} up to a normalized pair (by Remark \ref{remark-1}), taking test function $(u_0,v_0)$, we get $\norm{(u_0,v_0)}^p_{\W}= \int_{\Omega} \abs{u_0}^{\al} \abs{v_0}^{\be} \dX$. Hence 
\begin{align}\label{re2.1}
    F = \left( \frac{1}{p} - \frac{1}{q} \right) \left( \frac{\norm{(u_0,v_0)}^p_{\W}}{\left( \int_{\Omega} u_0^{\al} v_0^{\be} \dX \right)^{\frac{p }{q}}} \right)^{\frac{q}{q - p}} = \left( \frac{1}{p} - \frac{1}{q} \right)  \int_{\Omega} u_0^{\al} v_0^{\be} \dX.
\end{align}
Similarly, since $(Bw_0, C w_0)$ weakly solves \eqref{system_eqn_homogenous}, we get $$\left(B^p+C^p \right) [w_0]_{s,p}^p = B^{\al}C^{\be}  \int_{\Omega} w_0^{q} \dX,$$ and 
\begin{align*}
    F = \left( \frac{1}{p} - \frac{1}{q} \right) \left( \frac{B^p + C^p}{(B^{\al} C^{\be})^{\frac{p}{q}}} \frac{[w_0]_{s,p}^p}{ \left(\int_{\Omega} w_0^{q} \dX \right)^{\frac{p}{q}}} \right)^{\frac{q}{q - p}}= \left( \frac{1}{p} - \frac{1}{q} \right) B^{\al} C^{\be} \int_{\Omega} w_0^{q} \dX.
\end{align*}
Therefore, \eqref{re2.1} and the above identity conclude \eqref{re2}.  
\end{proof}

In the following proposition, we prove a non-existence result of the ground state solution. 

\begin{proposition}\label{not_attend}
Let $m=0$, i.e. $\al+\be = p^*_s$. Then $S_{\al, \be}$ has no minimizer on any open set $\Omega \subsetneq \rd$.
\end{proposition}

\begin{proof}
  Observe that $\|(\cdot,\cdot)\|_{\W}$ and $\left(\int_{\rd}|\cdot|^\al|\cdot|^\be\dx\right)^{\frac1{p_s^*}}$ are invariant under translation and scaling 
     $$(u,v)\mapsto R^{\frac{d-sp}{p}}(u(Rx+y),v(Rx+y)).$$
   Arguing as in \cite[Section I.4.5, pg. 42]{struwe}, $S_{\al,\be}$ is independent of $\Om$ in $\rd$. Suppose $(u_0,v_0)$ is a minimizer of $S_{\al, \be}$ on any open set $\Omega$ in $\rd$. Then the pair of functions $(\tilde{u}, \tilde{v}) \in \mathcal{W}(\rd)$ with $\tilde{u} = u_0, \tilde{v} =v_0$ in $\Omega$, and $\tilde{u}, \tilde{v} = 0$ in $\rd \setminus \Omega$, minimizes $S_{\al, \be}$ over $\rd$. Moreover, up to the normalized pair (as in Remark \ref{remark-1}), they satisfy
   \begin{equation}
\left\{\begin{aligned}
&(-\Delta_p)^s \tilde{u}= \frac{\al}{p^*_s} |\tilde{u}|^{\alpha-2}\tilde{u}|\tilde{v}|^{\beta} \;\;\text{in}\;\rd,\\
&(-\Delta_p)^s \tilde{v}= \frac{\beta}{p^*_s} |\tilde{v}|^{\beta-2} \tilde{v} |\tilde{u}|^{\alpha}\;\;\text{in}\;\rd, \\
&\tilde{u}, \tilde{v} \ge 0\;\;\text{in}\;\rd.
\end{aligned}
\right.
\end{equation}
 Applying the Strong maximum principle (see \cite[Proposition 3.7]{NR}), we see that either $\tilde{u} \equiv 0$ or $\tilde{u}>0$ a.e. in $\rd$ and either $\tilde{v} \equiv 0$ or $\tilde{v}>0$ a.e. in $\rd$, a contradiction, using the definition of $\tilde{u}, \tilde{v}$. This completes the proof. 
\end{proof}

\noi \textbf{Proof of Proposition \ref{existence&uniqueness}:} 
The proof of (i) follows by Proposition \ref{compact&existence} and (ii) follows by Proposition \ref{not_attend}.
\qed

For $m=0$, in view of the above non-existence result, we observe that any minimizing sequence of $S_{\al,\be}$ in $\Om$ can not have a convergent subsequence. However, in $\R^d$, we 
see that any minimizing sequence of $S_{\al,\be}$, up to translation and dilation, has a strongly convergent subsequence. A similar type of result related to the minimizing sequence of $S_{\al+\be}$ is proved in \cite{PaPi, FiMoSe} for $p=2$. 

Next, we recall the Morrey space. 
\begin{definition}[Morrey space]\label{Morreyspaces} For $r \in [1, \infty)$ and $\gamma \in [0, d]$, the homogeneous Morrey space $\mathcal{L}^{r,\gamma}(\rd)$ is defined as 
  \begin{equation*}
    \mathcal{L}^{r,\gamma}(\mathbb{R}^d) \coloneqq \left\{u:\rd\to\R \text{ measurable }: \sup_{x\in\rd,\, {R>0}}R^{\gamma}\fint_{B(x,R)} |u(y)|^r \dy < \infty \right\},
\end{equation*}
with norm
\begin{align*}
   \norm{u}_{\mathcal{L}^{r,\gamma}(\rd)}^r \coloneqq \sup_{x\in\rd,\, {R>0}}R^{\gamma}\fint_{B(x,R)} |u(y)|^r \dy. 
\end{align*}
\end{definition}

In view of \cite[Proposition 2.15]{Bi-Ch2026}, and using the fact that $W_0^{s,p}(\Omega) \subset \D^{s,p}(\rd)$, we have the following embedding. 

\begin{proposition}\label{interpolation(product)}
Let $\frac{p}{p^*_s} \le \theta <1$ and $p \le r \le p^*_s$. Then 
\begin{align}\label{Inter-morrey-sobolev-2}
\norm{(u,v)}_{L^{p^*_s}(\rd) \times L^{p^*_s}(\rd) } & \le C(d,p,s) \norm{(u,v)}_{\W}^{\theta} \no \\
&\norm{(u,v)}_{\mathcal{L}^{r, \frac{r(d-sp)}{p}}(\rd) \times \mathcal{L}^{r, \frac{r(d-sp)}{p}}(\rd)}^{1-\theta}, \; \forall \, (u,v) \in \W.
\end{align}
\end{proposition}

%\cite[Chapter 8]{willem}

\begin{proposition}\label{par_the_rescale}
   Let $m=0$, i.e. $\al+\be = p^*_s$. Let $\{(u_n,v_n)\}$ be a sequence in $\W$ such that $$\int_{\Om}{|u_n|^\alpha|v_n|^{\beta}}\dx=1, \; \forall \, n,\text{ and } \|(u_n,v_n)\|_{\W}^p\to S_{\alpha,\beta}, \text{ as } n \ra \infty.$$
    Then there exists $\{(y_n,r_n)\}\subset\R^d\times\R^+$ such that the sequence $$(\tilde{u}_n,\tilde{v}_n)=r_n^{\frac{d-sp}{p}}(u_n(r_nx+y_n),v_n(r_nx+y_n))$$ has a convergent subsequence, still denoted by $(\tilde{u}_n,\tilde{v}_n)$, such that $(\tilde{u}_n,\tilde{v}_n)\to (\tilde{u},\tilde{v})$ in $\W(\rd)$. Moreover, $r_n\to0$ and $y_n\to y$ in $\ov{\Om}$.
\end{proposition}

\begin{proof}
For $$(\tilde{u}_n,\tilde{v}_n)=r_n^{\frac{d-sp}{p}}({u}_n(r_nx+y_n),{v}_n(r_nx+y_n)),$$
we observe that $\norm{(\tilde{u}_n, \tilde{v}_n)}_{\W} = \|(u_n,v_n)\|_{\W}$. Therefore, $\{(\tilde{u}_n,\tilde{v}_n)\}$ is bounded in $\mathcal{W}(\rd)$. By the reflexivity of $\mathcal{W}(\rd)$, $(\tilde{u}_n,\tilde{v}_n)\rightharpoonup(\tilde{u},\tilde{v})$ in $\W(\rd)$.
Following the arguments as given in \cite[Step 2 of Theorem 1.2]{Bi-Ch2026} we first show that $(\tilde{u},\tilde{v})\neq(0,0)$ in $\W(\rd)$. By the hypothesis, for all $n$, $$1=\int_{\Om}{|u_n|^\alpha|v_n|^{\beta}}\dx\leq \|(u_n,v_n)\|^{p_s^*}_{L^{p_s^*}(\R^d)\times L^{p_s^*}(\R^d)}.$$
Therefore, $(u_n, v_n) \not \ra 0$ in $L^{p^*_s}(\rd) \times L^{p^*_s}(\rd)$, and 
\begin{align*}
    \inf_{n \in \mathbb{N}} \norm{(u_n, v_n)}_{L^{p^*_s}(\rd) \times L^{p^*_s}(\rd)}^{p^*_s} \ge 1.
\end{align*}
Hence, for $\frac{p}{p^*}\leq\theta<1$, the interpolation given in \eqref{Inter-morrey-sobolev-2} yields 
\begin{align*}
    1 \le C \norm{(u_n, v_n)}_{\W}^{\theta} \norm{(u_n, v_n)}_{\mathcal{L}^{p,d-sp}(\rd) \times \mathcal{L}^{p,d-sp}(\rd)}^{1- \theta} \le C_1 \norm{(u_n, v_n)}_{\mathcal{L}^{p,d-sp}(\rd) \times \mathcal{L}^{p,d-sp}(\rd)}^{1- \theta},
\end{align*}
which implies $\norm{(u_n,v_n)}_{\mathcal{L}^{p,d-sp}(\rd) \times \mathcal{L}^{p,d-sp}(\rd)} \ge C_2$ for some $C_2>0$, i.e., 
\begin{align*}
    \sup_{x \in \rd,\, R>0} R^{d-sp} \fint_{B(x,R)} \left( |u_n(y)|^p + |v_n(y)|^p \right) \dy \ge C_2.  
\end{align*}
Hence for every $n \in \N$, there exist $y_n \in \rd$ and $r_n \in (0, \infty)$ such that
\begin{align}\label{PSD-3}
    r_n^{d-sp} \fint_{B(y_n, r_n)} \left( |u_n(y)|^p + |v_n(y)|^p \right) \dy \ge \norm{(u_n, v_n)}_{\mathcal{L}^{p,d-sp}(\rd) \times \mathcal{L}^{p,d-sp}(\rd)}^p - \frac{C_2}{2n} > C_3,
\end{align}
for some $C_3>0$. By the change of variable and using the compact embedding of $\W \hookrightarrow L^p(\Omega) \times L^p(\Omega)$, we get 
\begin{align*}
    \omega_d C_3 \leq r_n^{-sp}  \int_{B(y_n, r_n)} \left( |u_n(y)|^p + |v_n(y)|^p \right) \dy & = \int_{B(0,1)} \left( |\tilde{u}_n(y)|^p + |\tilde{v}_n(y)|^p \right) \dy \\
    & \longrightarrow \int_{B(0,1)} \left( |\tilde{u}(y)|^p + |\tilde{v}(y)|^p \right) \dy, \text{ as } n \ra \infty.
\end{align*}
Therefore, $\tilde{u}, \tilde{v} \neq 0$.

By Brezis Lieb lemma, 
    \begin{align}
        1&=S_{\al,\be}^{-1}\lim_{n\to\infty}\|(\tilde{u}_n,\tilde{v}_n)\|_{\W}^p=S_{\al,\be}^{-1}\varlimsup_{n\to\infty}\|(\tilde{u}_n-\tilde{u},\tilde{v}_n-\tilde{v})\|_{\W}^p+S_{\al,\be}^{-1}\|(\tilde{u},\tilde{v})\|_{\W}^p\label{para1}\\
1&=\lim_{n\to\infty}\int_{\R^d}|\tilde{u}_n|^\al|\tilde{v}_n|^\be\dx=\lim_{n\to\infty}\int_{\R^d}|\tilde{u}_n-\tilde{u}|^\al|\tilde{v}_n-\tilde{v}|^\be\dx+\int_{\R^d}|\tilde{u}|^\al|\tilde{v}|^\be\dx\label{para2}.
    \end{align}
    By \eqref{para1} and the fact that $\tilde{u},\tilde{v}\neq0$,
    $$S_{\al,\be}^{-1}\|(\tilde{u},\tilde{v})\|_{\W}^p\in (0,1],\text{ and }S_{\al,\be}^{-1}\varlimsup_{n\to\infty}\|(\tilde{u}_n-\tilde{u},\tilde{v}_n-\tilde{v})\|_{\W}^p\in[0,1).$$
    Now, by the definition of $S_{\al,\be}$, from \eqref{para2}, 
    \begin{align*}
        1&\leq \left(S_{\al,\be}^{-1}\,\varlimsup_{n\to\infty}\|(\tilde{u}_n-\tilde{u},\tilde{v}_n-\tilde{v})\|_{\W}^p\right)^{\frac{p_s^*}p}+\left(S_{\al,\be}^{-1}\|(\tilde{u},\tilde{v})\|_{\W}^p\right)^{\frac{p_s^*}p}\\
        &\leq S_{\al,\be}^{-1}\varlimsup_{n\to\infty}\|(\tilde{u}_n-\tilde{u},\tilde{v}_n-\tilde{v})\|_{\W}^p+S_{\al,\be}^{-1}\|(\tilde{u},\tilde{v})\|_{\W}^p=1.
    \end{align*}
    Here, we use the fact that for $a\in[0,1]$ and $t\geq1$, $a^t\leq a$.  
    Thus, all the previous inequalities are equal. But the equality in $a^t\leq a$, where $a\in[0,1]$ and $t\geq1$, can happen only if $a=0$ or $1$. Hence, we conclude that $\|(\tilde{u},\tilde{v})\|_{\W}^p=S_{\al,\be}$ and $\varlimsup_{n\to\infty}\|(\tilde{u}_n-\tilde{u},\tilde{v}_n-\tilde{v})\|_{\W}^p=0$. As a result, $(\tilde{u}_n,\tilde{v}_n)\to(\tilde{u},\tilde{v})$ in $\W(\rd)$.
    
Next, we prove the claims about $y_n$ and $r_n$. From \eqref{PSD-3} for large $n$,
$$0<C_3=r_n^{d-sp}\fint_{B(y_n,r_n)} (|{u}_n|^p+|{v}_n|^p)\dx \le Cr_n^{-sp}\int_{\Om} (|{u}_n|^p+|{v}_n|^p)\dx.$$ If $r_n\to\infty$, then using the fact that $\{u_n\}, \{v_n\}$ are bounded in $L^p(\Om)$, we get a  contradiction. Thus, $\{r_n\}$ is a bounded sequence. Since $u_n,v_n$ vanish in $\rd \setminus \Om$, using \eqref{PSD-3}, we have $y_n\in\{x\in\R^d:d(x,\ov{\Om})\leq r_n\}$. Therefore, $\{y_n\}$ is bounded in $\R^d$. Suppose, $r_n\to r_0>0$ and $y_n\to y\in\R^d$. Notice that $\tilde{u}_n$ and $\tilde{v}_n$ are supported on $\frac{\Om-y_n}{r_n}$ and as $n\to\infty$, $$\frac{\Om-y_n}{r_n}\to \frac{\Om-y}{r_0}.$$ Since $(\tilde{u}_n,\tilde{v}_n)\to(\tilde{u},\tilde{v})$ in $\W(\rd)$, $(\tilde{u},\tilde{v})$ satisfies, 
$$\int_{\frac{\Om-y}{r_0}}{|\tilde{u}|^\alpha|\tilde{v}|^{\beta}}\dx=1, \text{ and } \|(\tilde{u},\tilde{v})\|_{\W}^p=S_{\alpha,\beta}.$$
Thus, $(\tilde{u},\tilde{v})$ minimizes $S_{\alpha,\beta}$ on a bounded set, which is absurd. Hence $r_n\to0$ and since $y_n\in\{x\in\R^d:d(x,\ov{\Om})\leq r_n\}$, we must have $y\in\ov{\Om}$.
\end{proof}

\section{Concentration compactness principle associated with the homogeneous system}\label{conc_comp}
For $u \in L_{\text{loc}}^1(\rd)$ and a.e. $x \in \rd$ we consider the $(s,p)$-gradient of $u$ as 
\begin{equation*}
    \abs{D^su}^p(x) := \int_{\rd} \frac{\abs{u(x)-u(y)}^p}{\abs{x-y}^{d+sp}} \, \dy. 
\end{equation*}
In \cite[Theorem 2.5]{SuMa} (for $m=0$) and \cite[Theorem 1.1]{AlPa} (for $m>0$), the authors proved a concentration compactness principle to get the exact behaviour of weakly convergent sequences of $\wps$ in the space of measures. More preciously, they have shown that if $u_n \rightharpoonup u$ in $\wps$, then there exist two non-negative measures $\sigma,\nu$ and at most countable set $\{x_j\}_{j \in \mathcal{J}} \subset \overline{\Omega}$ such that as $n \ra \infty$,
\begin{align*}
    \abs{D^su_n}^p \overset{\ast}{\rightharpoonup} \sigma, \frac{\abs{u_n}^{\psm}}{\abs{x}^m} \dx \overset{\ast}{\rightharpoonup} \nu, 
\end{align*}
where 
\begin{align}\label{cc-recall-1}
    \sigma \ge \abs{D^su}^p + \sum_{j \in \mathcal{J}} \sigma_j \delta_{x_j}, \, \sigma_j = \sigma(\{ x_j \}), \nu = \frac{\abs{u}^{\psm}}{\abs{x}^m} \dx + \sum_{j \in \mathcal{J}} \nu_j \delta_{x_j}, \, \nu_j = \nu(\{ x_j \}),
\end{align}
with $\sigma_j \ge S_{\al+\be} \left( \nu_j \right)^{\frac{p}{\psm}}$, where $\delta_{x_j}(E) = 1$ if $x_j \in E$ and $\delta_{x_j}(E) = 0$ if $x_j \not \in E$ for every measurable set $E \subset \rd$.  Further, in the subcritical case $m>0$, they showed that $\{ x_j \}_{j \in \mathcal{J}} = \{0\}$. Motivated by these results, we provide a concentration compactness principle to study the behaviour of the map $G$. We would like to mention that the boundedness of $\Omega$ plays a vital role. For $m=0$, Lu and Shen studied a similar concentration compactness principle in \cite[Lemma 5.1]{LuShen2020}.

Now, we recall the definition of the tightness of measures.  
\begin{definition}
 Let $\mathcal{S}$ be a $\sigma$-algebra of subsets of $\rd$ and  $\{ \mu_n \}$ be a sequence of measures defined on $\mathcal{S}$. The collection $\{ \mu_n \}$ is called `tight', if for every $\var>0$, there exists $K_{\var} \Subset \rd$ such that
    \begin{equation*}
        \sup_{n} \mu_n (\rd \setminus K_{\var}) < \var. 
    \end{equation*}
\end{definition}

The following lemma states the convergences of some integrals, which are direct consequences of the Brezis-Lieb lemma (see \cite[Theorem 1]{brezis1983relation}).

\begin{lemma}[Brezis-Lieb lemma]\label{Brezis-Lieb} 
Let  $0\le m \le sp < d$ and $\al, \be > 1$ be such that $\al + \be = p^*_s(m)$. Let $\{ (u_n, v_n) \}$ weakly converges to $ {(u,v)}$ in $\mathcal{W}$. Then up to a subsequence, the following identities hold:  
\begin{enumerate}
    % \item[{\rm(i)}] $\norm{u_n}_{\wps}^p = \norm{u_n - u}_{\wps}^p + \norm{u}_{\wps}^p + o_n(1)$, and $\norm{v_n}_{\wps}^p = \norm{v_n - v}_{\wps}^p + \norm{v}_{\wps}^p + o_n(1)$.
    \item[{\rm(i)}] For every $\phi \in \wps$, 
    \begin{align*}
        &\int_{\rd} \abs{D^s u_n}^p \phi \dx = \int_{\rd} \abs{D^s (u_n-u)}^p \phi \dx + \int_{\rd} \abs{D^s u}^p \phi \dx + o_n(1), \text{ and } \\
        &\int_{\rd} \abs{D^s v_n}^p  \phi \dx = \int_{\rd} \abs{D^s (v_n-v)}^p \phi \dx + \int_{\rd} \abs{D^s v}^p \phi \dx + o_n(1).
    \end{align*}
    \item[{\rm(ii)}] For every $\phi \in \wps$, 
    \begin{align*}
        \int_{\Omega} \abs{u_n}^{\al}\abs{v_n}^{\be} \phi \frac{\dx}{\abs{x}^m} =\int_{\Omega} \left|\abs{u_n}^{\al}\abs{v_n}^{\be} - \abs{u}^{\al}\abs{v}^{\be}\right| \phi \frac{\dx}{\abs{x}^m} + \int_{\Omega} \abs{u}^{\al}\abs{v}^{\be} \phi \frac{\dx}{\abs{x}^m} + o_n(1).
    \end{align*}
    \item[{\rm(iii)}] It holds 
    \begin{align*}
        [u_n]_{s,p}^p = [u_n-u]_{s,p}^p + [u]_{s,p}^p + o_n(1), \text{ and } [v_n]_{s,p}^p = [v_n-v]_{s,p}^p + [v]_{s,p}^p + o_n(1).
    \end{align*}
\end{enumerate}
\end{lemma}

\begin{proposition}[Concentration compactness principle]\label{CCP}
    Let $\Omega \subset \R^d$ be a bounded open set containing $0$. Let $0\le m \le sp$ and $\al, \be > 1$ be such that $\al + \be = p^*_s(m)$. Let $\{ u_n \}, \{ v_n \}$ be sequences in $\wps$ such that $u_n \rightharpoonup u$ and $v_n \rightharpoonup v$ in $\wps$. Then there exist measures $\sigma_1, \sigma_2,\nu_1, \nu_2, \nu$ and at most countable sets $\{ x_j \}_{j \in \mathcal{J}_1}$, $\{ x_j \}_{j \in \mathcal{J}_2}$, $\{ x_j \}_{j \in \mathcal{J}_3}$, such that up to subsequences,
    \begin{align*}
        & \abs{D^su_n}^p \overset{\ast}{\rightharpoonup} \sigma_1, \abs{D^s v_n}^p \overset{\ast}{\rightharpoonup} \sigma_2, \frac{\abs{u_n}^{\psm}}{\abs{x}^m} \dx \overset{\ast}{\rightharpoonup} \nu_1,  \frac{\abs{v_n}^{\psm}}{\abs{x}^m} \dx \overset{\ast}{\rightharpoonup} \nu_2, \text{ and } \frac{\abs{u_n}^{\al} \abs{v_n}^{\be}}{\abs{x}^m} \dx \overset{\ast} {\rightharpoonup} \nu,
    \end{align*}
    in the dual space of $\C_b(\rd)$, where 
    \begin{align*}
        &\sigma_1 \ge \abs{D^su}^p + \sum_{j \in \mathcal{J}_1} (\sigma_1)_j \delta_{x_j}, \; \sigma_2 \ge \abs{D^sv}^p + \sum_{j \in \mathcal{J}_2} (\sigma_2)_j \delta_{x_j}, \\
        & \nu_1 = \frac{\abs{u}^{\psm}}{\abs{x}^m} \dx + \sum_{j \in \mathcal{J}_1} (\nu_1)_j \delta_{x_j}, \; \nu_2 = \frac{\abs{v}^{\psm}}{\abs{x}^m} \dx + \sum_{j \in \mathcal{J}_2} (\nu_2)_j \delta_{x_j}, \text{ and } \\
        & \nu = \frac{\abs{u}^{\al} \abs{v}^{\be}}{\abs{x}^m} \dx +  \sum_{j \in \mathcal{J}_3} \nu_j \delta_{x_j},
    \end{align*}
    and the following relation holds
    \begin{align*}
       &(\sigma_1)_j \ge S_{\al + \be} ((\nu_1)_j)^{\frac{p}{p_s^*(m)}}, \; \forall \, j \in \mathcal{J}_1, \\
       &(\sigma_2)_j \ge S_{\al + \be} ((\nu_2)_j)^{\frac{p}{p_s^*(m)}},  \; \forall \, j \in \mathcal{J}_2, \text{ and } \\
       &(\sigma_1)_j+ (\sigma_2)_j \ge S_{\al, \be} (\nu_j)^{\frac{p}{p_s^*(m)}},  \; \forall \, j \in \mathcal{J}_3.
    \end{align*}
    Further, if $m>0$ then each set $\{ x_j \}_{j \in \mathcal{J}_1}, \{ x_j \}_{j \in \mathcal{J}_2}$ and $\{ x_j \}_{j \in \mathcal{J}_3}$ is $\{0\}$. Define
    \begin{align*}
        (\sigma_1)_{\infty} := \lim_{R \ra \infty} \lim_{n \ra \infty} \int_{\abs{x}>R} \abs{D^s u_n}^p \dx, \text{ and } (\sigma_2)_{\infty} := \lim_{R \ra \infty} \lim_{n \ra \infty} \int_{\abs{x}>R} \abs{D^s v_n}^p \dx.
    \end{align*}
    If
    $\abs{D^s(u_n-u)}^p \overset{\ast}{\rightharpoonup} \d \tilde{\sigma_1}$, $\abs{D^s(v_n-v)}^p \overset{\ast}{\rightharpoonup} \d \tilde{\sigma_2}$, then the following identities hold:
    \begin{align*}
        &\lim_{n \ra \infty} \int_{\rd} \abs{D^su_n}^p \dx = (\sigma_1)_{\infty} + \int_{\rd} \abs{D^su}^p  \dx  + \norm{\tilde{\sigma_1}}, \text{ and} \\
        &\lim_{n \ra \infty} \int_{\rd} \abs{D^sv_n}^p  \dx = (\sigma_2)_{\infty} + \int_{\rd} \abs{D^s v}^p \dx+ \norm{\tilde{\sigma_2}},
    \end{align*}
    where $\norm{\tilde{\sigma_1}} = \int_{\rd} \d \tilde{\sigma_1}, \text{ and } \norm{\tilde{\sigma_2}} = \int_{\rd} \d \tilde{\sigma_2}.$
\end{proposition}

\begin{proof}
  For $\{ u_n \}$ and $\{ v_n \}$ as given in the hypothesis, consider the following measures: 
 \begin{align*}
      &(\sigma_1)_n := \abs{D^su_n}^p \dx, \; (\sigma_2)_n := \abs{D^sv_n}^p \dx, \text{ and } \nu_n := \frac{\abs{u_n}^{\al} \abs{v_n}^{\be}}{\abs{x}^m} \dx. 
 \end{align*}
In view of \eqref{cc-recall-1}, it is enough to analyze the weak star convergence of the sequence of measures $\{ \nu_n \}$.
Since $\{ u_n \}$ and $\{ v_n \}$ are bounded in $\wps$, using \eqref{HS} and \eqref{HS2} it follows that
\begin{align*}
      &(\sigma_1)_n(\rd) = [u_n]_{s,p}^p \le C, \; (\sigma_2)_n(\rd) = [v_n]_{s,p}^p \le C, \text{ and } \nu_n(\Omega) \le C [u_n]^{\al}_{s,p} [v_n]^{\be}_{s,p} \le C.
\end{align*}
Thus the sequences $\{ (\sigma_1)_n \}$, $\{ (\sigma_2)_n \}$, and $\{ \nu_n \}$ are uniformly bounded. Now, we verify the tightness of these sequences. Since $u_n=v_n=0$ in $\rd \setminus \Omega$, 
\begin{align*}
\nu_n(\rd \setminus \Omega) = 0,
\end{align*}
for each $n \in \N$. Let $\ep>0$ and $a>1$. Choose $R_{\ep}>1$ such that  $C(d,s,p) (a')^{d+sp} \omega_d R_{\ep}^{-sp} < \ep$ and $\Omega \subset B_{R_{\ep}}(0)$ and then choose $U_{\ep} \subset \rd$ such that $B_{aR_{\ep}}(0) \subset U_{\ep}$ and
$\text{dist}(\rd \setminus U_{\ep},B_{aR_{\ep}}(0)) > 0$. If $x \in \rd \setminus U_{\ep}$, then $\abs{x} = bR_{\ep}$ where $b>a$. Now for $x \in \rd \setminus U_{\ep}$ and $y \in \Omega$, we have 
\begin{equation*}
    \abs{x-y}^{d+sp} \ge \abs{\abs{x} - \abs{y}}^{d+sp} \ge \abs{ \abs{x} - R_{\ep}}^{d+sp} = \abs{(b-1)R_{\ep}}^{d+sp} = \left( \frac{\abs{x}}{b'} \right)^{d+sp},
\end{equation*}
Hence using the embedding $\wps \hookrightarrow L^p(\Omega)$, 
\begin{equation*}
    \abs{D^su_n}^p(x) = \int_{\rd} \frac{\abs{u_n(y)}^p}{\abs{x-y}^{d+sp}} \, \dy \le (b')^{d+sp} \int_{\Omega} \frac{\abs{u_n(y)}^p}{\abs{x}^{d+sp}} \, \dy \le \left( \frac{b'}{\abs{x}} \right)^{d+sp} C [u_n]_{s,p}^p \le  \left( \frac{b'}{\abs{x}} \right)^{d+sp} C, 
\end{equation*}
where $C=C(d,s,p)$ is the embedding constant. Using the above estimate, we see that for each $n \in \N$,
\begin{align*}
    (\sigma_1)_n(\rd \setminus U_{\ep}) \le C (b')^{d+sp} \int_{\rd \setminus U_{\ep}} \frac{\dx}{\abs{x}^{d+sp}} & \le C (b')^{d+sp} \int_{\rd \setminus B_{R_{\ep}}(0)} \frac{\dx}{\abs{x}^{d+sp}} \\
    & = C (b')^{d+sp} \omega_d {R_{\ep}}^{-sp} \\
    & \le C (a')^{d+sp} \omega_d {R_{\ep}}^{-sp} < \ep. 
\end{align*}
This follows that $\{ (\sigma_1)_n \}$ is tight. Using similar arguments, we also get $\{ (\sigma_2)_n \}$ is tight. Therefore, applying Prokhorov’s theorem, there exist measures $\sigma_1, \sigma_2, \nu$ such that
\begin{align}\label{cc1}
    (\sigma_1)_n \overset{\ast}{\rightharpoonup} \sigma_1, \;(\sigma_2)_n \overset{\ast}{\rightharpoonup} \sigma_2, \text{ and } \frac{\abs{u_n}^{\al} \abs{v_n}^{\be}}{\abs{x}^m} \dx  \overset{\ast}{\rightharpoonup} \nu,
\end{align}
respectively in the dual space of $\C_b(\R^d)$. In view of \cite[Theorem 2.5]{SuMa} (for $m=0$) and \cite[Theorem 1.1]{AlPa} (for $m>0$), we have 
\begin{align*}
    \sigma_1 \ge \abs{D^su}^p + \sum_{j \in \mathcal{J}_1} (\sigma_1)_j \delta_{x_j}, \; \sigma_2 \ge \abs{D^sv}^p + \sum_{j \in \mathcal{J}_2} (\sigma_2)_j \delta_{x_j},
\end{align*}
where if $m>0$ then each set $\{ x_j \}_{j \in \mathcal{J}_1}$ and $\{ x_j \}_{j \in \mathcal{J}_2}$ is $\{0\}$.

\noi \textbf{Step 1:} In this step we show that 
\begin{align}\label{III}
    \nu_n \overset{\ast} {\rightharpoonup} \nu + \frac{\abs{u}^{\al}\abs{v}^{\be}}{\abs{x}^m}, \text{ where } \nu = \sum_{j \in \mathcal{J}_3} \nu_j \delta_{x_j}, \text{ and for $m>0$, $\{ x_j \}_{j \in \mathcal{J}_3} = \{0\}$. }
\end{align}
We first assume $u=v=0$. Take $\phi \in \C_c^{\infty}(\rd)$ and $\phi \ge 0$. Let $\theta>0$. From the definition of  $S_{\al, \be}$ and the inequality given in \cite[(2.1)]{SuMa}, we get 
\begin{align}\label{cc2}
    &S_{\al, \be} \left( \int_{\rd} \frac{\abs{u_n \phi}^{\al} \abs{v_n \phi}^{\beta}}{\abs{x}^m} \dx \right)^{\frac{p}{\psm}} \le [u_n \phi]_{s,p}^p + [v_n \phi]_{s,p}^p \no \\
    & \le (1+\theta)\int_{\rd} \phi^p \left( \abs{D^s u_n}^p + \abs{D^s v_n}^p  \right) \dx + C_{\theta} \int_{\rd} \abs{D^s \phi}^p \left( \abs{u_n}^p + \abs{v_n}^p \right) \dx. 
\end{align}
for some $C_{\theta} > 0$. Further, using \cite[Lemma 2.3]{SaSu}, $\abs{D^s \phi}^p \in L^{\infty}(\rd)$. Moreover, by the compact embeddings of $\wps$ into $L^p(\Omega)$ and the fact that $u_n,v_n \in \wps$, we get $u_n \ra 0$ and $v_n \ra 0$ in $L^p(\rd)$. Thus
 \begin{align*}
     \int_{\rd} \abs{D^s \phi}^p \left( \abs{u_n}^p + \abs{v_n}^p \right) \dx \le \norm{\abs{D^s \phi}^p}_{{L^{\infty}(\rd)}} \left( \norm{u_n}^p_{L^p(\rd)} + \norm{v_n}^p_{L^p(\rd)} \right) \ra 0, \text{ as } n \ra \infty.
 \end{align*}
Therefore, taking the limit as $n \ra \infty$ in \eqref{cc2} and using \eqref{cc1} for every $\theta > 0$ we obtain 
\begin{align*}
   S_{\al, \be} \left( \int_{\rd} \phi^{p^*_s(m)} \, \d \nu \right)^{\frac{p}{\psm}} \le (1+\theta) \int_{\rd} \phi^p (\d \sigma_1 + \d \sigma_2).
\end{align*}
Taking $\theta \ra 0$ in the above estimate, 
\begin{align}\label{cc3}
   S_{\al, \be} \left( \int_{\rd} \phi^{p^*_s(m)} \, \d \nu \right)^{\frac{p}{\psm}} \le \int_{\rd} \phi^p (\d \sigma_1 + \d \sigma_2).
\end{align}
Now we consider the non-negative measure $\sigma = \sigma_1 + \sigma_2$ to get from \eqref{cc3} that 
\begin{align*}
    S_{\al, \be} \left( \int_{\rd} \abs{\phi}^{p^*_s(m)} \, \d \nu \right)^{\frac{p}{\psm}} \le \int_{\rd} \abs{\phi}^p \d \sigma, \; \forall \, \phi \in \C_c^{\infty}(\rd).
\end{align*}
The above estimate holds on $L^p(\sigma)$ by the density argument. Now we apply \cite[Lemma 1.2]{LionsI} to obtain $\nu = \sum_{j \in \mathcal{J}_3} \nu_j \delta_{x_j}$. Further, for $m>0$, we choose $\phi \in \C_c^{\infty}(\rd)$, $\phi \ge 0$ such that $0 \not \in \text{supp}(\phi)$. Further, using \eqref{HS2}, we have 
\begin{align*}
    &\int_{\rd} \phi^{p^*_s(m)} \frac{\abs{u_n}^{\al} \abs{v_n}^{\beta}}{\abs{x}^m} \dx \\
    &\le C(d,s,p) \norm{\phi}^{p^*_s(m)}_{L^{\infty}(\rd)} \left( \int_{\Omega \, \cap \, \text{supp}(\phi)} \frac{\abs{u_n}^{p^*_s(m)}}{\abs{x}^m} \dx \right)^{\frac{\al}{p^*_s(m)}} \left( \int_{\Omega \, \cap \, \text{supp}(\phi)} \frac{\abs{v_n}^{p^*_s(m)}}{\abs{x}^m} \dx \right)^{\frac{\be}{p^*_s(m)}} \\
    &\le C(d,s,p) \norm{\phi}^{p^*_s(m)}_{L^{\infty}(\rd)}  \norm{\abs{x}^{-m}}_{L^{\infty}(\text{supp}(\phi))} \norm{u_n}_{L^{p^*_s(m)}(\Omega)}^{\al} \norm{v_n}_{L^{p^*_s(m)}(\Omega)}^{\be}. 
\end{align*}
By the compact embedding of $\wps \hookrightarrow L^{p^*_s(m)}(\Omega)$, and using the above estimate we see that 
\begin{align*}
  \int_{\rd} \phi^{p^*_s(m)} \, \d \nu = \lim_{n \ra \infty}  \int_{\rd} \phi^{p^*_s(m)} \frac{\abs{u_n}^{\al} \abs{v_n}^{\beta}}{\abs{x}^m} \dx = 0,  \text{ for every } \phi \ge 0 \text{ in } \C_c^{\infty}(\rd) \text{ with } 0 \not \in \text{supp}(\phi),
\end{align*}
which implies that $\{ x_j \}_{j \in \mathcal{J}_3} = \{0\}$. Suppose there exists some $x_j \in \mathcal{J}_3$ with $x_j \neq 0$ and $x_j \in \text{supp}(\nu)$, then choosing $\phi \in \cc(\Omega)$ and $\phi \equiv 1$ in a neighbourhood of $x_j$ with support does not contain $0$, we get a contradiction repeating our previous calculation.  Next, for $u\neq 0$ and $v \neq 0$, taking $\tilde{u}_n = u_n - u$ and $\tilde{v}_n = v_n - v$, we immediately get  
\begin{align}\label{nonzerolimit1}
    \frac{\abs{\tilde{u}_n}^{\al} \abs{\tilde{v}_n}^{\be}}{\abs{x}^m} \dx \overset{\ast}{\rightharpoonup} \nu = \sum_{j \in \mathcal{J}} \nu_j \delta_{x_j}, 
\end{align}
where for $m>0$, $\{ x_j \}_{j \in \mathcal{J}_3} = \{0\}$. Take $\phi \in \C_b(\rd)$. Applying Lemma \ref{convergence2},
\begin{align*}
    &\int_{\rd} \left( \abs{u_n}^{\alpha} \abs{v_n}^{\beta} - \abs{u}^{\alpha} \abs{v}^{\beta} - \abs{\tilde{u}_n}^{\alpha} \abs{\tilde{v}_n}^{\beta} \right) \frac{\dx}{\abs{x}^m} = o_n(1).
\end{align*}
The above identity yields 
\begin{align*}
    &\left| \int_{\rd} \left|\abs{u_n}^{\al}\abs{v_n}^{\be} - \abs{u}^{\al}\abs{v}^{\be}\right| \phi \frac{\dx}{\abs{x}^m}- \int_{\rd} \abs{\tilde{u}_n}^{\al} \abs{\tilde{v}_n}^{\be} \phi \frac{\dx}{\abs{x}^m} \right| \\
    &\le \norm{\phi}_{L^{\infty}(\rd)} \int_{\rd} \left| \abs{u_n}^{\al} \abs{v_n}^{\be}- \abs{u}^{\al}\abs{v}^{\be} - \abs{\tilde{u}_n}^{\al} \abs{\tilde{v}_n}^{\be}\right| \frac{\dx}{\abs{x}^m}  = o_n(1).
\end{align*}
Therefore, using \eqref{nonzerolimit1} we obtain 
\begin{align}\label{id-2}
 \lim_{n \ra \infty} \int_{\rd} \left|\abs{u_n}^{\al}\abs{v_n}^{\be} - \abs{u}^{\al}\abs{v}^{\be}\right| \phi \frac{\dx}{\abs{x}^m} = \lim_{n \ra \infty} \int_{\rd} \abs{\tilde{u}_n}^{\al} \abs{\tilde{v}_n}^{\be} \phi \frac{\dx}{\abs{x}^m} = \int_{\rd} \phi \d \nu, 
\end{align}
where $\nu = \sum_{j \in \mathcal{J}_3} \nu_j \delta_{x_j}$, and for $m>0$, $\{ x_j \}_{j \in \mathcal{J}_3} = \{0\}$. Further, by Lemma \ref{Brezis-Lieb}, 
\begin{align*}
    \int_{\rd} \abs{u_n}^{\al}\abs{v_n}^{\be} \phi \frac{\dx}{\abs{x}^m} =\int_{\rd} \left|\abs{u_n}^{\al}\abs{v_n}^{\be} - \abs{u}^{\al}\abs{v}^{\be}\right| \phi \frac{\dx}{\abs{x}^m} + \int_{\rd} \abs{u}^{\al}\abs{v}^{\be} \phi \frac{\dx}{\abs{x}^m} + o_n(1).
\end{align*}
Clearly, the above identity and \eqref{id-2} immediately give \eqref{III}.

\noi \textbf{Step 2:} Now for a sequence $(u_n, v_n) \rightharpoonup (u,v)$ in $\mathcal{W}$, the compact embeddings of $\wps \hookrightarrow L^p(\Omega)$ yields $u_n \ra u$ and $v_n \ra v$ in $L^p(\Omega)$. Further, using $\abs{D^s \phi}^p \in L^{\infty}(\rd)$ (\cite[Lemma 2.3]{SaSu}), we get 
\begin{align*}
    \lim_{n \ra \infty} \int_{\Omega} \abs{D^s \phi }^p \abs{u_n}^p \dx = \int_{\Omega} \abs{D^s \phi }^p \abs{u}^p \dx \Longrightarrow \lim_{n \ra \infty} \int_{\rd} \abs{D^s \phi }^p \abs{u_n}^p \dx = \int_{\rd} \abs{D^s \phi }^p \abs{u}^p \dx.
\end{align*}
Similarly, we also have $\lim_{n \ra \infty} \int_{\rd} \abs{D^s \phi }^p |v_n|^p \dx = \int_{\rd} \abs{D^s \phi }^p |v|^p \dx.$ Therefore, taking the limit as $n \ra \infty$ in \eqref{cc2}, we obtain the following estimate for every $\phi \in \C_c^{\infty}(\rd)$,
\begin{align}\label{cc6}
    S_{\al, \be} \left( \int_{\rd} \phi^{p^*_s(m)} \, \d \nu \right)^{\frac{p}{\psm}} \le  (1+\theta)\int_{\rd} \phi^p (\d \sigma_1 + \d \sigma_2) + C_{\theta} \int_{\rd} \abs{D^s \phi}^p \left( \abs{u}^p + \abs{v}^p \right) \dx. 
\end{align}
For $j \in \mathcal{J}_3$, choose $\{ x_j \}$. Then choose $\phi \in \C_c^{\infty}(\rd)$ such that $0 \le \phi \le 1$, $\phi(x_j) = 1$ and $\text{supp}(\phi) = B_{1}(x_j)$.  For $\ep >0$, consider $\phi_{\ep}(x) = \phi(\frac{x}{\ep})$. From the fact that $\nu \ge \nu_j \delta_{x_j}$ we obtain from \eqref{cc6} that 
\begin{align}\label{cc7}
    S_{\al, \be} \left(\nu_j\right)^{\frac{p}{\psm}} \le  (1+ \theta) \sigma(B_{\ep}(x_j)) + C_{\theta} \int_{\rd} \abs{D^s \phi_{\ep}}^p \left( \abs{u}^p + \abs{v}^p \right) \dx,
\end{align}
where $\nu_j = \nu (\{ x_j\})$. Further, using a similar set of arguments as given in the proof of \cite[Theorem 1.1 (pg. 436)]{AlPa}, we obtain 
\begin{align*}
    \lim_{\ep \ra 0} \int_{\rd} \abs{D^s \phi_{\ep}}^p \left( \abs{u}^p + \abs{v}^p \right) \dx = 0.
\end{align*}
Hence taking the limit as $\ep \ra 0$ in \eqref{cc7} we get 
\begin{align*}
    S_{\al, \be} \left(\nu_j\right)^{\frac{p}{\psm}} \le  (1+ \theta) \sigma(\{ x_j \}) = (1+ \theta) \left( \sigma_1(\{ x_j \}) + \sigma_2(\{ x_j \}) \right), \; \forall \, \theta >0. 
\end{align*}
Now taking $\theta \ra 0$ we obtain $(\sigma_1)_j+ (\sigma_2)_j \ge S_{\al, \be} (\nu_j)^{\frac{p}{p_s^*(m)}}$. 

\noi \textbf{Step 3:} For $R>0$, set $\psi_R \in \mathcal{C}_b(\rd)$, defined as $\psi_R(x) = 1$ if $\abs{x}>R$ and $\psi_R(x) = 0$ if $\abs{x} \le \frac{R}{2}$. We write 
\begin{align}\label{cc8}
    \lim_{n \ra \infty} \int_{\rd} \abs{D^s u_n}^p \dx = \lim_{n \ra \infty} \left( \int_{\rd} \psi_R(x) \abs{D^s u_n}^p  \dx + \int_{\rd} \left( 1-\psi_R(x) \right) \abs{D^s u_n}^p \dx \right),
\end{align}
where using Lemma \ref{Brezis-Lieb}, 
\begin{align*}
   \lim_{n \ra \infty} \int_{\rd} \left( 1-\psi_R(x) \right) \abs{D^s u_n}^p \dx & = \lim_{n \ra \infty} \int_{\rd}  \left( 1-\psi_R(x) \right) \abs{D^s (u_n-u)}^p \dx  \\ & + \int_{\rd} \left( 1-\psi_R(x) \right) \abs{D^s u}^p \dx.
   \end{align*}
Therefore, \eqref{cc8} yields  
\begin{align}\label{cc9}
    \lim_{n \ra \infty} \int_{\rd} & \abs{D^s u_n}^p \dx  = \lim_{n \ra \infty} \int_{\{ \abs{x} > R\}} \abs{D^s u_n}^p \dx \no \\
    &+ \lim_{n \ra \infty} \int_{\rd} \left( 1-\psi_R(x) \right) \abs{D^s (u_n-u)}^p \dx + \int_{\rd} \left( 1-\psi_R(x) \right) \abs{D^s u}^p \dx.
\end{align}
Further, using the dominated convergence theorem, 
\begin{align*}
    \lim_{R \ra \infty} \int_{\rd} \left( 1-\psi_R(x) \right) \abs{D^s u}^p \dx = \int_{\rd} \abs{D^s u}^p \dx.
\end{align*}
Moreover, since $1-\psi_R \in \C_b(\rd)$, we get
\begin{align*}
    \lim_{n \ra \infty} \int_{\rd} \left( 1-\psi_R(x) \right) \abs{D^s (u_n-u)}^p \dx = \int_{\rd} \left( 1-\psi_R(x) \right) \d \tilde{\sigma}_1.
\end{align*}
Again, using the dominated convergence theorem,
\begin{align*}
    \lim_{R \ra \infty} \lim_{n \ra \infty} \int_{\rd} \left( 1-\psi_R(x) \right) \abs{D^s (u_n-u)}^p \dx = \int_{\rd} \d \tilde{\sigma}_1 
\end{align*}
Therefore, taking the limit as $R \ra \infty$ in \eqref{cc9} we obtain 
\begin{align*}
    \lim_{n \ra \infty} \int_{\rd} \abs{D^su_n}^p \dx = (\sigma_1)_{\infty} + \int_{\rd} \abs{D^su}^p  \dx  + \int_{\rd} \d \tilde{\sigma_1}.
\end{align*}
The above identity for the sequence $\{ v_n \}$ similarly holds. This completes the proof.    
\end{proof}

\section{Existence of a nontrivial solution with least energy}\label{exst_nontriv}
For brevity, for $\al+\be = \psm$, we set the following quantity: 
\begin{align*}
    \tilde{c} = \tilde{c}(\al, \be,d,s,m,\gamma):= \left(\frac{1}{p}-\frac{1}{p_{s}^{*}(m)}\right) \frac{S_{\al,\be}^{\frac{d- m}{sp-m}}}{\gamma^{\frac{p}{p_{s}^{*}(m)-p}}}.
\end{align*}  
We also recall $c_0$ defined in \eqref{c0}. Observe that $c_0<\infty$ since $p_{s}^{*}(m)\ge p$. 
\begin{definition}
A sequence $\{(u_n,v_n)\}$ in $\W$ is said to be a PS sequence for $I$ 
 at the level $c$ if $$I(u_n,v_n) \ra c\text{ in }\R,\quad I'(u_n,v_n) \ra 0\text{ in }\mathcal{W}^*\text{ as }n\to\infty.$$We say that $I$ satisfies PS condition at the level $c$ (in short, $\text{(PS)}_c$ condition), if $\{(u_n,v_n)\}$ is a PS sequence for $I$ at the level $c$ then $\{(u_n,v_n)\}$ has a convergent subsequence in $\W$.
\end{definition}

In the following lemma, depending on the values of $r$ and $m$, we see that the energy functional $I$ (defined in \eqref{energy-main}) satisfies the $\text{(PS)}_c$ condition under a certain range of $c$. 
\begin{lemma}[$\text{(PS)}_c$ condition]\label{ps1}
Let $0 \leq m \leq sp, p\leq r <p_{s}^{*}$ and $\max \left\{r, p_{s}^{*}(m)\right\}>p$. Let $\al, \be > 1$ be such that $\al+\be = \psm$. Then the following are true:
\begin{enumerate}
	\item[{\rm(i)}] If $p_{s}^{*}(m)>p$ and $r \in\left(p, p_{s}^{*}\right)$, then for $\gamma,\eta>0$, $I$ satisfies $\text{(PS)}_c$ condition for every $c<\tilde{c}$.
	\item[{\rm(ii)}] If $p_{s}^{*}(m)>p$ and $r=p$ then for $\gamma>0$ and $\eta < \frac{S_{\al+\be}}{c_0}$, $I$ satisfies $\text{(PS)}_c$ condition
    for every $c<\tilde{c}$.
    \item[{\rm(iii)}]  If $p_{s}^{*}(m)=p$ and $r \in\left(p, p_{s}^{*}\right)$, then for $\gamma<S_{\al,\be}$ and $\eta>0$, $I$ satisfies $\text{(PS)}_c$ condition for every $c \in \mathbb{R}$.
\end{enumerate}
\end{lemma}
\begin{proof}
Let $\left\{\left(u_{n}, v_{n}\right)\right\}$ be a PS sequence for  $I$ at the level $c$, i.e.,
\begin{align}\label{p1}
	I\left(u_{n}, v_{n}\right)=\frac{1}{p}\left\|\left(u_{n}, v_{n}\right)\right\|_{\W}^{p}-\frac{\gamma}{p_{s}^{*}(m)} \int_{\Omega} \frac{\left|u_{n}\right|^{\alpha}\left|v_{n}\right|^{\beta}}{|x|^{m}} \dx -\frac{\eta}{r} \int_{\Omega}\left( | u_{n}|^{r}+|v_{n}|^{r}\right)   \dx =  c+o_{n}(1)
\end{align}
and
$I^{\prime}\left(u_{n}, v_{n}\right)(\phi, \psi) \rightarrow 0$ for all $\phi, \psi\in \wps$. In particular, taking $(\phi, \psi)=\left(u_{n}, v_{n}\right)$, 
\begin{equation}\label{p2}
	I^{\prime}\left(u_{n}, v_{n}\right)\left(u_{n}, v_{n}\right)=\|(u_{n}, v_{n})\|_{\W} ^{p}-\gamma \int_{\Omega} \frac{\left|u_{n}\right|^{\alpha}\left|v_{n}\right|^{\beta}}{|x|^{m}}  \dx-\eta\int_{\Omega}\left( |u_{n}|^{r}+|v_{n}|^{r}\right)   \dx.
\end{equation}
By \eqref{p1} and \eqref{p2}, we get
\begin{align}\label{p3}
\notag	& C+o_{n}(1)\left(\left[u_{n}\right]_{s, p}+\left[v_{n}\right]_{s, p}\right) \geq p 
    I\left(u_{n}, v_{n}\right)-I^{\prime}\left(u_{n}, v_{n}\right)\left(u_{n},  v_{n}\right) \\
	& =\gamma\left(1-\frac{p}{p_{s}^{*}(m)}\right) \int_{\Omega} \frac{\left|u_{n}\right|^{\alpha}\left|v_{n}\right|^{\beta}}{|x|^{m}}  \dx+\eta\left(1-\frac{p}{r}\right)\int_{\Omega}(\left|u_{n}\right|^{r}+\left|v_{n}\right|^{r})  \dx.
\end{align}
\noi \textbf{Boundedess of $\text{(PS)}_c$ sequence:} If $p_{s}^{*}(m)>p$ and $r>p$ then by \eqref{p3}, we obtain
\begin{align}\label{identities-1}
\begin{split}
&\int_{\Omega} \frac{\left|u_{n}\right|^{\alpha}\left|v_{n}\right|^{\beta}}{|x|^{m}}  \dx \leq C\left(1+\left[u_{n}\right]_{s, p}+\left[v_{n}\right]_{s, p}\right), \text{ and }  \\
& \int_{\Omega}\left(\left|u_{n}\right|^{r}+\left|v_{n}\right|^{r}\right)  \dx \leq C\left(1+\left[u_{n}\right]_{s, p}+\left[v_{n}\right]_{s, p}\right).
\end{split}
\end{align}
By using above estimates and \eqref{p1}, we have
$$
\frac{1}{p}\|(u_{n}, v_{n})\|_{\W} ^{p}\leq C\left(1+\left[u_{n}\right]_{s, p}+\left[v_{n}\right]_{s, p}\right),
$$
which implies that $\left\{\left(u_{n}, v_{n}\right)\right\}$ is bounded in $\W$.

\noindent If $r=p$ then by the assumption that $\max \{p_{s}^{*}(m), r\}>p$, we have $p_{s}^{*}(m)>p$. Now by using  \eqref{p3}, we have
$$
\int_{\Omega} \frac{\left|u_{n}\right|^{\alpha}\left|v_{n}\right|^{\beta}}{|x|^{m}}  \dx \leq C\left(1+\left[u_{n}\right]_{s, p}+\left[v_{n}\right]_{s, p}\right).
$$
By the H\"{o}lder's inequality and using the definition of $S_{\al+\be}$,
\begin{align*}
   \int_{\Omega}|u_n|^p \dx\leq \int_{\Omega}\frac{|u_n|^p}{|x|^{\frac{pm}{p_{s}^{*}(m)}}}|x|^{\frac{pm}{p_{s}^{*}(m)}} \dx & \leq \left(\int_{\Omega} \frac{\left|u_{n}\right|^{p_{s}^{*}(m)}}{|x|^{m}} \dx\right)^{\frac{p}{p_s^{*}(m)}}\left(\int_{\Omega}|x|^{\frac{pm}{p_{s}^{*}(m)-p}} \dx\right)^{\frac{p_{s}^{*}(m)-p}{p_{s}^{*}(m)}} \\
   &\leq c_0S_{\al+\be}^{-1} \left[u_{n}\right]_{s,p}^{p}. 
\end{align*}
Similarly, we also have
$
\int_{\Omega}|v_n|^p  \dx\leq c_0 S_{\al+\be}^{-1}\left[v_{n}\right]_{s,p}^{p}.
$
By using the above information and \eqref{p1}, we have
$$
	\left(\frac{1}{p}-\frac{c_{0} S_{\al+\be}^{-1}\eta}{p}\right)\left\|\left(u_{n}, v_{n}\right)\right\|_{\W}^{p}\leq  C\left(1+\left[u_{n}\right]_{s, p}+\left[v_{n}\right]_{s, p}\right)
$$
which implies the boundedness of  $\left\{\left(u_{n}, v_{n}\right)\right\}$  in $\W$ as $\eta < \frac{S_{\al+\be}}{c_{0}}$.

\noindent  If $p_{s}^{*}(m)=p$ (i.e., $m=sp$) then we have $r>p$. From the definition of $S_{\al, \be}$, 
\begin{align}\label{def-use-2}
    \int_{\Omega} \frac{\left|u_{n}\right|^{\alpha}\left|v_{n}\right|^{\beta}}{|x|^{m}} \dx \leq S_{\al,\be}^{-1}\|\left(u_{n}, v_{n}\right)\|_{\W}^{p}.
\end{align}
By \eqref{def-use-2}, \eqref{identities-1}, and \eqref{p1}, we get
\begin{align*}
    \left(\frac{1}{p}-\frac{\gamma S_{\al,\be}^{-1}}{p}\right)\left\|\left(u_{n}, v_{n}\right)\right\|_{\W}^{p} \leq C\left(1+\left[u_{n}\right]_{s, p}+\left[v_{n}\right]_{s, p}\right),
\end{align*}
which implies $\left\{\left(u_{n}, v_{n}\right)\right\}$ is bounded in $\W$ as $\gamma <S_{\al,\be}$.

\noindent \textbf{Convergence of $\text{(PS)}_c$ sequence:} Since $\mathcal{W}$ is reflexive, upto a subsequence, $\{ (u_n,v_n)\}$ weakly converges to $(u,v)$ in $\mathcal{W}$. Hence $u_n \ra u$ and $v_n \ra v$ a.e. in $\Omega$. By \cite[Lemma 2.5]{Bi-Ch2026}, we have the following convergence for every $(\phi, \psi) \in \W$:
\begin{align*}
    &\lim_{n \ra \infty} \int_{\Omega} \abs{u_n(x)}^{\al -2} u_n(x) \abs{v_n(x)}^{\beta} \phi(x) \frac{\dx}{\abs{x}^m} = \int_{\Omega} \abs{u(x)}^{\al -2} u(x) \abs{v(x)}^{\beta} \phi(x) \frac{\dx}{\abs{x}^m}, \\ & \lim_{n \ra \infty} \int_{\Omega} \abs{u_n(x)}^{\al} \abs{v_n(x)}^{\be -2} v_n(x) \psi(x) \frac{\dx}{\abs{x}^m} = \int_{\Omega} \abs{u(x)}^{\al} \abs{v(x)}^{\be -2} v(x) \psi(x) \frac{\dx}{\abs{x}^m},\\
    &\lim_{n \ra \infty} \iint\limits_{\rd\times\rd}\frac{|u_n(x)-u_n(y)|^{p-2}(u_n(x)-u_n(y))(\phi(x)-\phi(y))}{|x-y|^{d+sp}} \dxy  \\
    &\quad\quad=\iint\limits_{\rd\times\rd}\frac{|u(x)-u(y)|^{p-2}(u(x)-u(y))(\phi(x)-\phi(y))}{|x-y|^{d+sp}} \dxy, \text{ and } \\
    &\lim_{n \ra \infty} \iint\limits_{\rd\times\rd}\frac{|v_n(x)-v_n(y)|^{p-2}(v_n(x)-v_n(y))(\psi(x)-\psi(y))}{|x-y|^{d+sp}} \dxy  \\
    &\quad\quad=\iint\limits_{\rd\times\rd}\frac{|v(x)-v(y)|^{p-2}(v(x)-v(y))(\psi(x)-\psi(y))}{|x-y|^{d+sp}} \dxy.
    \end{align*}
Also, by the compact embeddings of $\W \hookrightarrow L^r(\Omega) \times L^r(\Omega)$, we get $u_n \ra u$ and $v_n \ra v$ in $L^r(\Omega)$, which further give
\begin{align*}
    \lim_{n \ra \infty} \int_{\Omega} \abs{u_n}^{r-2} u_n \phi = \int_{\Omega} \abs{u}^{r-2} u \phi, \text{ and } \lim_{n \ra \infty} \int_{\Omega} \abs{v_n}^{r-2} v_n \psi = \int_{\Omega} \abs{v}^{r-2} v \psi,
\end{align*}
for every $(\phi, \psi) \in \W$. Therefore, using $I'(u_n,v_n)(\phi, \psi) = o_n(1)$ and above convergences, we conclude $I'(u,v)(\phi, \psi) = 0$ for every $(\phi, \psi) \in \W$. In particular, we now choose $\phi=u, \psi=v$. Therefore,
$$ I(u, v)=\gamma\left(\frac{1}{p} - \frac{1}{p_{s}^{*}(m)}\right) \int_{\Omega} \frac{|u|^{\alpha}|v|^{\beta}}{|x|^{m}}  \dx+ \eta\left( \frac{1}{p}-\frac{1}{r}\right) \int_{\Omega}\left(|u|^{r}+|v|^{r}\right) \dx \ge 0.$$
Further, applying  the Brezis-Lieb lemma (Lemma \ref{Brezis-Lieb}) and Lemma \ref{convergence2}, we get
\begin{align*}
    \begin{split}
    o_{n}(1)&=I^{\prime}\left(u_{n}, v_{n}\right)\left(u_{n}, v_{n}\right)-I^{\prime}(u, v)(u, v) \\
	& =\left\|\left(u_{n}, v_{n}\right)\right\|_{\W}^{p}-\|(u, v)\|_{\W}^{p}-\gamma \int_{\Omega} \frac{\left|u_{n}\right|^{\alpha}\left|v_{n}\right|^{\beta} - \left|u\right|^{\alpha}\left|v\right|^{\beta}}{|x|^{m}} \dx - \eta \int_{\Omega} \left( \abs{u_n}^r + \abs{v_n}^r \right) \dx\\
    &\quad \quad +\eta \int_{\Omega} \left( \abs{u}^r + \abs{v}^r \right) \dx \\
    & = \left\|\left(u_{n}-u, v_{n}-v\right)\right\|_{\W}^{p}-\gamma \int_{\Omega} \frac{\left|u_{n}-u\right|^{\alpha}\left|v_{n}-v\right|^{\beta}}{|x|^{m}} \dx - \eta \int_{\Omega} \left( \abs{u_n-u}^r + \abs{v_n-v}^r \right) \dx.
    \end{split} 
\end{align*}
From the compact embeddings of $\W \hookrightarrow L^r(\Omega) \times L^r(\Omega)$, the above identities yield 
\begin{align}\label{p4}
    o_{n}(1) = \left\|\left(u_{n}-u, v_{n}-v\right)\right\|_{\W}^{p}-\gamma \int_{\Omega} \frac{\left|u_{n}-u\right|^{\alpha}\left|v_{n}-v\right|^{\beta}}{|x|^{m}} \dx.
\end{align}
Again using Lemma \ref{Brezis-Lieb}, Lemma \ref{convergence2}, and compact embeddings of $\W \hookrightarrow L^r(\Omega) \times L^r(\Omega)$, we get
\begin{equation}\label{p5}
I\left(u_{n}, v_{n}\right)-I(u, v)=\frac{1}{p}\left\|\left(u_{n}-u, v_{n}-v\right)\right\|_{\W}^{p}-\frac{\gamma}{p_{s}^{*}(m)} \int_{\Omega} \frac{\left|u_{n}-u\right|^{\alpha}\left|v_{n}-v\right|^{\beta}}{|x|^{m}}  \dx+o_{n}(1),
\end{equation}
where using \eqref{p4} we write
\begin{align}\label{p6}
\begin{split}
&\frac{1}{p}\left\|\left(u_{n}-u, v_{n}-v\right)\right\|_{\W}^{p}-\frac{\gamma}{p_{s}^{*}(m)} \int_{\Omega} \frac{\left|u_{n}-u\right|^{\alpha}\left|v_{n}-v\right|^{\beta}}{|x|^{m}}  \dx \\
&=\left(\frac{1}{p}-\frac{1}{p_{s}^{*}(m)}\right)\left\|\left(u_{n}-u, v_{n}-v\right)\right\|_{\W}^{p}+o_{n}(1).
\end{split}
\end{align}
Further, using \eqref{p5},
\begin{align}\label{p7}
	\frac{1}{p}\left\|\left(u_{n}-u, v_{n}-v\right)\right\|_{\W}^{p}  -\frac{\gamma}{p_{s}^{*}(m)} \int_{\Omega} \frac{\left|u_{n}-u\right|^{\alpha}\left|v_{n}-v\right|^{\beta}}{|x|^{m}} \dx \leq I\left(u_{n}, v_{n}\right) \leq c+o_{n}(1).
\end{align}
If $p_{s}^{*}(m)>p$ and $c<\tilde{c}$, then \eqref{p6} and \eqref{p7} yield
\begin{equation}\label{p8}
	\left(\frac{1}{p}-\frac{1}{p_{s}^{*}(m)}\right)\left\|\left(u_{n}-u, v_{n}-v\right)\right\|_{\W}^{p}< \left(\frac{1}{p}-\frac{1}{p_{s}^{*}(m)}\right) \frac{S_{\al,\be}^{\frac{d - m}{sp-m}}}{\gamma^{\frac{p}{p_{s}^{*}(m)-p}}}.
\end{equation}
By \eqref{p4}, using the definition of $S_{\al,\be}$ and then using \eqref{p8}, we obtain
$$
\begin{aligned}
	o_{n}(1)& \ge \left\|\left(u_{n}-u, v_{n}-v\right)\right\|_{\W}^{p}-\gamma S_{\al,\be}^{\frac{-p_{s}^{*}(m)}{p}}\left\|\left(u_{n}-u, v_{n}-v\right)\right\|_{\W}^{p_{s}^{*}(m)} \\
	& =\left\|\left(u_{n}-u, v_{n}-v\right)\right\|_{\W}^{p}\left(1-\gamma S_{\al,\be}^{\frac{-p_{s}^{*}(m)}{p}}\left\|\left(u_{n}-u, v_{n}-v\right)\right\|_{\W}^{p_{s}^{*}(m)-p}\right) \\
    &\ge \kappa\left\|\left(u_{n}-u, v_{n}-v\right)\right\|_{\W}^{p},
\end{aligned}
$$
for some $\kappa>0$ which implies $u_{n} \rightarrow u$ and $v_{n} \rightarrow v$ in $\wps$.

\noi If $p_{s}^{*}(m)=p$ i.e., $m=sp$, by \eqref{p4} and using the definition of $S_{\al,\be}$ we have
$$
\begin{aligned}
		o_{n}(1) =\left\|\left(u_{n}-u, v_{n}-v\right)\right\|_{\W}^{p}-\gamma \int_{\Omega}\frac{\left|u_{n}-u\right|^{\alpha}\left|v_{n}-v\right|^{\beta}}{|x|^{sp}} \dx \geq\left\|\left(u_{n}-u, v_{n}-v\right)\right\|_{\W}^{p}\left(1-\gamma S_{\al, \be}^{-1} \right).
\end{aligned}
$$
As $S_{\al,\be}>\gamma$, we immediately get $u_{n} \rightarrow u$ and $v_{n} \rightarrow v$ in $\wps$. This completes the proof. 
\end{proof}
\begin{remark}\label{radial-function}
    Assume $m<sp$, i.e. $p_s^*(m)>p$. By \cite[Theorem 1.1]{SaSu}, we choose a positive radially symmetric decreasing minimizer $U_m=U_m(r)$ for $S_{\al+\be}$ which weakly solves
    \begin{equation}\label{bubble_eqn}
        (-\Delta_p)^sU_m=\frac{U_m^{p_s^*(m)-1}}{|x|^m}\text{ in }\R^d.
    \end{equation}
    By \cite[Lemma 2.9]{chen20183065}, there exist $l>1$ such that $$U_m(l r)\leq \frac12 U_m(r),\; \forall \, r\geq1.$$ For $\var>0$, the function $U_{m,\var}(x)=\var^{-\frac{N-ps}{p}}U_m\left(\frac{x}{\var}\right)$ is also a minimizer of $S_{\al+\be}$ satisfying \eqref{bubble_eqn}. Now, for $\delta > 0$, we define the radially symmetric non-negative decreasing function 
    \begin{equation}\label{trunc_bubble}
        u_{m,\var,\delta}(r):=\begin{cases}
            U_{m,\var}(r),&r\leq\delta;\\
            \displaystyle \frac{U_{m,\var}(\delta)(U_{m,\var}(r)-U_{m,\var}(l\delta))}{U_{m,\var}(\delta)-U_{m,\var}(l\delta)},&\delta\leq r\leq l\delta;\\
            0,&r\geq l\delta.
        \end{cases}
    \end{equation}
    Observe that $u_{m,\var,\delta}\in W_0^{s,p}(\Om)$ for every $\delta<l^{-1}\text{dist}(0,\pa\Om)$.
\end{remark}
Next, we define the following min-max level associated with the energy functional:
$$c_{1}=\inf_{\gamma\in \varGamma}\max_{t\in [0,1]}I(\gamma(t)), \text{ where } \varGamma=\{\gamma\in \mathcal{C}([0,1],\W):\gamma(0)=(0,0),\,I(\gamma(1))<0\}.$$ 

\begin{lemma}[Mountain pass level]\label{ps2} Let $0 \leq m <sp$ and $\al, \be > 1$ be such that $\al+\be = \psm$. Then $c_1 < \tilde{c}$ holds for any $r$ satisfying
\begin{align}\label{mp1_1}
\begin{cases}
r>p^{*}_s-p^{\prime}, & \text { if } d<p^{2} s;\\
r \geq p, & \text { if } d \geq p^2s.
\end{cases}
\end{align}
\end{lemma}
\begin{proof}
First, observe that \eqref{mp1_1} implies 
$r \ge \frac{p^*_s}{p'}$ for every $d,p,s$. So it is enough to prove $c_1 < \tilde{c}$ whenever $r \ge \frac{p^*_s}{p'}$. For simplicity, we assume $\delta=1$ but the same proof works for any $\delta>0$. For sufficiently small $\varepsilon \in (0,1)$, we consider the radially symmetric non-negative decreasing function $u_{m, \varepsilon,1}$ given in \eqref{trunc_bubble}. In view of Proposition \ref{relation1}, we consider the following two functions in $\W$:
\begin{align}\label{def}
 u_{0}=\alpha^{\frac{1}{p}} u_{m, \varepsilon, 1}, \text{ and } v_{0}=\beta^{\frac{1}{p}} u_{m, \varepsilon, 1}.   
\end{align}
Using the norm estimates of $u_{m, \varepsilon, 1}$ in \cite[Lemma 2.10]{chen20183065}, we obtain
$$
\begin{aligned}
	 I\left(t u_{0}, t v_{0}\right):=g_{\varepsilon}(t)&=\frac{1}{p} \left(\left[t u_{0}\right]_{s, p}^{p}+\left[t v_{0}\right]_{s, p}^{p} \right) - \frac{\gamma}{p_{s}^{*}(m)} \int_{\Omega} \frac{\left(t u_{0}\right)^{\alpha}\left(t v_{0}\right)^{\beta}}{|x|^{m}}\dx\\
    &\quad \quad-\frac{\eta}{r}\int_{\Omega}\left(\left(t u_{0}\right)^{r}+\left(t v_{0}\right)^{r}\right)\dx \\
	& =\frac{t^{p}}{p}(\alpha+\beta)\left[u_{m,\ep,1}\right]_{s, p}^{p}-\gamma\frac{ t^{p_{s}^{*}(m)}}{p_{s}^{*}(m)}\left(\alpha^{\frac{\alpha}{p}}\beta^{\frac{\beta}{p}}\right) \int_{\Omega}\frac{u_{m, \varepsilon,1}^{p_{s}^{*}(m)}}{|x|^m}\dx\\
    &\quad \quad-\frac{\eta t^{r}}{r}\left(\alpha^{\frac{r}{p}}+\beta^{\frac{r}{p}}\right)\int_{\Omega} u_{m, \varepsilon,1}^{r}\dx\\
	& \leq \frac{t^{p}}{p}(\alpha+\beta)\left(S_{\al+\be}^{\frac{d-m}{sp-m}}+ C \varepsilon^{\frac{d-p s}{p-1}}\right) - \gamma \frac{t^{p_{s}^{*}(m)}}{p_{s}^{*}(m)}\left(\alpha^{\frac{\alpha}{p}} \beta^{\frac{\beta}{p}}\right)\left(S_{\al+\be}^{\frac{d-m}{p s-m}}-C\varepsilon^{\frac{d-m}{p-1}}\right) \\
	&\quad \quad-\frac{\eta t^{r}}{r}\left(\alpha^{\frac{r}{p}}+\beta^{\frac{r}{p}}\right) h_{r}(\varepsilon),
\end{aligned}$$
where $C>0$ is a constant and
\begin{align*}
h_{r}(\varepsilon) = \begin{cases} \varepsilon^{\frac{d}{p}}|\log \varepsilon|, & \text { if } r= \frac{p^{*}_s}{p'};\\ 
\varepsilon^{d-\frac{d-sp}{p}r}, & \text { if } r>\frac{p^{*}_s}{p'}.\end{cases}
\end{align*}
Observe that, $g_{\ep}(0)= 0$, $g_{\ep}(t)>0$ for $t>0$ small enough. Further, there exists $C>0$ (large enough) such that
\begin{align}\label{g_ep_1}
    g_{\varepsilon}(t) \leq \left(\frac{t^{p}}{p}(\alpha+\beta) -\gamma\frac{ t^{p_{s}^{*}(m)}}{p_{s}^{*}(m)}\left(\alpha^{\frac{\alpha}{p}}\beta^{\frac{\beta}{p}}\right)\right)S_{\al+\be}^{\frac{d-m}{p s-m}} + C t^p\varepsilon^{\frac{d-p s}{p-1}}+ Ct^{p_{s}^{*}(m)} \varepsilon^{\frac{d-m}{p-1}}-\frac{t^{r}}{C} h_{r}(\varepsilon).
\end{align}
Hence, there exists $\varepsilon_{0}>0$ (small enough) such that for every $\ep \in [0, \ep_0]$, $g_{\varepsilon}(t) \rightarrow-\infty$ as $t \rightarrow \infty$. 
Therefore, by the continuity of $g_{\ep}$, there exist $m_1>m_2>0$ such that
$$\sup\limits_{t \geq 0} g_{\varepsilon}(t)=\sup\limits_{t \in[m_1,m_2]} g_{\varepsilon}(t) \quad \forall \varepsilon \in\left[0, \varepsilon_{0}\right].$$
Consider the function
$$ h(t)=\frac{t^{p}}{p}(\alpha+\beta)-\gamma \frac{t^{p_{s}^{*}(m)}}{p_{s}^{*}(m)} \alpha^{\frac{\alpha}{p}}\beta^{\frac{\beta}{p}}, \; t>0.$$
Note that $h$ attains its maximum at
$$t^{*}=\frac{1}{\gamma^{\frac{1}{p_{s}^{*}(m)-p}}}\left(\frac{\alpha+\beta}{\alpha^{\frac{\alpha}{p}}\beta^{\frac{\beta}{p}}}\right)^{\frac{1}{p_{s}^{*}(m)-p}},$$
which concludes that 
$$
\begin{aligned}
\sup\limits_{t \geq 0} h(t)&=\frac{(t^{*})^{p}}{p}(\alpha+\beta)-\gamma\frac{ (t^{*})^{p_{s}^{*}(m)}}{p_{s}^{*}(m)} \alpha^{\frac{\alpha}{p}}\beta^{\frac{\beta}{p}} \\
	& =\frac{\gamma^{\frac{p}{p-p_{s}^{*}(m)}}}{p}\left(\frac{\alpha+\beta}{\alpha^{\frac{\alpha}{p}}\beta^{\frac{\beta}{p}}}\right)^{\frac{p}{p_{s}^{*}(m)-p}}(\alpha+\beta)-\frac{\gamma}{p_{s}^{*}(m)}\gamma^{\frac{p_{s}^{*}(m)}{p-p_{s}^{*}(m)}}\left(\frac{\alpha+\beta}{\alpha^{\frac{\alpha}{p}}\beta^{\frac{\beta}{b}}}\right)^{\frac{p_{s}^{*}(m)}{p_{s}^{*}(m)-p}}\alpha^{\frac{\alpha}{p}}\beta^{\frac{\beta}{b}}\\
	& =\gamma^{\frac{p}{p-p_{s}^{*}(m)}} \frac{(\alpha+\beta)^{\frac{p_{s}^{*}(m)}{p_{s}^{*}(m)-p}}}{\alpha^{\frac{\alpha}{p_{s}^{*}(m)-p}} \beta^{\frac{\beta}{p_{s}^{*}(m)-p}}}\left(\frac{1}{p}-\frac{1}{p_{s}^{*}(m)}\right) \\
	& =\gamma^{\frac{p}{p-p_{s}^{*}(m)}}\left(\left(\frac{\alpha}{\beta}\right)^{\frac{\beta}{\alpha+\beta}}+\left(\frac{\beta}{\alpha}\right)^{\frac{\alpha}{\alpha+\beta}}\right)^{\frac{d-m}{p s-m}} \left(\frac{1}{p}-\frac{1}{p_{s}^{*}(m)}\right).
\end{aligned}
$$
Therefore, using \eqref{g_ep_1} and the relation between $S_{\al,\be}$ and $S_{\al+\be}$ (Proposition \ref{relation1}), we obtain
$$
\begin{aligned}
	&\sup_{t \in[m_1,m_2]} g_{\varepsilon}(t) \leq S_{\al+\be}^{\frac{d-m}{ps-m}}\sup_{t\geq 0} h(t) +\sup_{t \in[m_1,m_2]}\left(Ct^{p} \varepsilon^{\frac{d-ps}{p-1}}+C t^{p_{s}^{*}(m)} \varepsilon^{\frac{d-m}{p-1}}-\frac{1}{C} t^{r} h_{r}(\varepsilon)\right) \\
	& \leq \gamma^{\frac{p}{p-p_{s}^{*}(m)}}\left(\left(\frac{\alpha}{\beta}\right)^{\frac{\beta}{\alpha+\beta}}+\left(\frac{\beta}{\alpha}\right)^{\frac{\alpha}{\alpha+\beta}}\right)^{\frac{d-m}{p s-m}} \left(\frac{1}{p}-\frac{1}{p_{s}^{*}(m)}\right)S_{\al+\be}^{\frac{d-m}{p s-m}} +C \varepsilon^{\frac{d-p s}{p-1}}+C \varepsilon^{\frac{d-m}{p-1}}-\frac{1}{C} h_{r}(\varepsilon) \\
	& =\tilde{c}+C \varepsilon^{\frac{d-p s}{p-1}}+C \varepsilon^{\frac{d-m}{p-1}}-\frac{1}{C} h_{r}(\varepsilon).
\end{aligned}
$$
Next, using \cite[Lemma 3.1]{chen20183065}, we have 
$$C\varepsilon^{\frac{d-p s}{p-1}}+C \varepsilon^{\frac{d-m}{p-1}}-\frac{1}{C} h_{r}(\varepsilon)<0,$$
for sufficiently small $\varepsilon$, so that
\begin{align}\label{en-1}
    \sup_{t\geq 0}I\left(t u_{0}, t v_{0}\right)=\sup_{t\geq 0}g_{\varepsilon}(t) =\sup_{t \in[m_1,m_2]} g_{\varepsilon}(t) <\tilde{c}.
\end{align}
Finally, consider the path $\gamma_0 \in \mathcal{C}([0,1],\W)$ defined as $\gamma_0(t)= (tu_0,tv_0)$. Therefore, by \eqref{en-1}, we conclude $c_1 < \tilde{c}$.
\end{proof}

\begin{lemma} [Mountain pass geometry]\label{ps3}
Let $0 \leq m \leq sp, p\leq r <p_{s}^{*}$, $\max \left\{r, p_{s}^{*}(m)\right\}>p$ and 
\begin{equation*}
    \begin{cases}
        0<\eta<\lambda_1,&\text{ if }r=p;\\
        \eta>0,&\text{ if }r \in \left(p, p^{*}_s\right).
    \end{cases}
\end{equation*}
Then the following are true:

\begin{enumerate}
\item[{\rm{(a)}}] There exist $\rho,\Theta>0$ such that
$I(u, v)>\Theta $  on the sphere $ S= \left\{ (u,v) \in \mathcal{W} : \norm{(u,v)}_{\W} = \rho\right\}$.
\item [{\rm{(b)}}] There exists $\left(u_{1}, v_{1}\right) \in \W$  with $\norm{(u_{1}, v_{1})}_{\W} > \rho$ such that $I\left(u_{1}, v_{1}\right)<0.$ 
\end{enumerate}
\end{lemma}
\begin{proof}
\noi (a) Using \eqref{HS2} and \eqref{lm_1} notice that 
\begin{equation*}
    I(u, v) \geq \begin{cases}
       \displaystyle \left(\frac{1}{p}-\frac{\eta}{p\lambda_{1}}-C\|(u, v)\|_{\W}^{p_{s}^{*}(m)-p}\right)\|(u, v)\|_{\W}^{p},&\text{ if }r=p; \\
        \left(\frac{1}{p}-C\|(u, v)\|_{\W}^{p_{s}^{*}(m)-p}-C\|(u, v)\|_{\W}^{r-p}\right)\|(u, v)\|_{\W}^{p},&\text{ if }r \in \left(p, p^{*}_s\right).
    \end{cases}
\end{equation*}
Therefore, there exist $\rho,\Theta>0$ such that
$I(u, v)>\Theta$, if $(u,v) \in S$.

\noi (b) For $\left(u, v\right) \in \W\setminus\{(0,0)\}$,
$$
\begin{aligned}
& I\left(t u, t v\right)=\frac{t^{p}}{p}\left(\left[u\right]_{s, p}^{p}+\left[v\right]_{s, p}^{p}\right)- \frac{\gamma t^{p_{s}^{*}(m)}}{p_{s}^{*}(m)} \int_{\Omega} \frac{\left|u\right|^{\alpha}\left|v\right|^{\beta}}{|x|^{m}}\dx-\frac{\eta t^{r}}{r} \int_{\Omega}\left(\left|u\right|^{r}+\left|v\right|^{r}\right)\dx.
\end{aligned}
$$
By using the fact that $\max \left\{p_{s}^{*}(m), r\right\}>p$ we have $I\left(t u, tv\right) \rightarrow-\infty$ as $t \rightarrow \infty$, which proves the existence of $(u_1,v_1)$. 
\end{proof}
Now, we are ready to prove the existence of the least energy solution. A weak solution $(\tilde{u},\tilde{v})\in \W$ of \eqref{system_eqn} has the least energy if
$$l_{S}:=\displaystyle\inf_{(u,v)\in S}I(u,v)=I(\tilde{u},\tilde{v}),$$ where $S$ is the set of all critical points of the functional $I$. We define the Nehari manifold
\begin{align}\label{nehari}
 \mathcal{N}=\left\{(u,v)\in \W\setminus\{(0,0)\}: I^{\prime}(u,v)(u,v)=0\right\}.    \end{align}
Also, denote $$l_{\mathcal{N}}:=\displaystyle\inf_{(u,v)\in \mathcal{N}}I(u,v).$$

\begin{proposition}\label{subexistence&uniqueness}
    Let $0 \leq m \le sp$, $p \le r < p^*_s$, and $\max \left\{r, p_{s}^{*}(m)\right\}>p$. Let $\al,\be>1$ be such that $\al + \be = p^*_s(m)$.  Choose $\eta, \gamma>0$ satisfying
\begin{align*}
    &\eta < \min\left\{\lambda_1, \frac{S_{\al+\be}}{c_{0}}\right\}, \text{ when } m<sp, r=p, \text{ and } \gamma < S_{\al, \be}, \text{ when } m=sp, r>p.
\end{align*}
Then \eqref{system_eqn} admits a nontrivial solution $(\tilde{u}, \tilde{v})$ with the least energy provided
\begin{align*}
\begin{cases}r \geq p, & \text { if } d \geq p^{2} s; \\ r>p^{*}_s-p^{\prime}, & \text { if } d<p^{2} s.\end{cases}
\end{align*}
\end{proposition}

\begin{proof}
  By Lemmata \ref{ps1}, \ref{ps2} and \ref{ps3}, and applying the mountain pass theorem \cite[Theorem 2.1]{Ambrosetti1973}, there exists a critical point $(\tilde{u},\tilde{v})\in \W$ for the functional $I$ which implies $(\tilde{u},\tilde{v})$ is a solution for \eqref{system_eqn}. Further, $I(\tilde{u},\tilde{v})=c_1>0$ implies that $(\tilde{u},\tilde{v})$ is a nontrivial solution. Next, we prove that $(\tilde{u},\tilde{v})$ has the least energy, i.e., $I(\tilde{u},\tilde{v})=l_{S}$. By the definition of $l_\mathcal{N}$ and $l_S$, we have 
\begin{equation}\label{min_1}
l_\mathcal{N}\leq l_{S}\leq c_1=I(\tilde{u},\tilde{v}).
\end{equation}
Define the function $k:\R^{+}\rightarrow\R$ as $$k(t)=I(tu,tv)=\frac{t^{p}}{p}\left(\left[u\right]_{s, p}^{p}+\left[v\right]_{s, p}^{p}\right)- \frac{\gamma t^{p_{s}^{*}(m)}}{p_{s}^{*}(m)} \int_{\Omega} \frac{\left|u\right|^{\alpha}\left|v\right|^{\beta}}{|x|^{m}}\dx-\frac{\eta t^{r}}{r} \int_{\Omega}\left(\left|u\right|^{r}+\left|v\right|^{r}\right)\dx.$$ We observe that $k'(t)=0$ if and only if $(tu,tv)\in\mathcal{N}$. Consequently, 
$k(t)>0$ for sufficiently small $t$ and $k(t)<0$ for sufficiently large $t$. Further observe that $k'(t)=0$ if and only if $$[u]_{s,p}^p+[v]_{s,p}^p=t^{p_s^\ast(m)-p}\gamma\int_{\Omega} \frac{\left|u\right|^{\alpha}\left|v\right|^{\beta}}{|x|^{m}}\dx+t^{r-p}\eta \int_{\Omega}\left(\left|u\right|^{r}+\left|v\right|^{r}\right)\dx.$$ Since $p_s^\ast(m)-p,\, r-p\geq0$, the right-hand side in the above is a strictly increasing function. Thus, there exists a unique $t_{0}\in(0,\infty)$ such that $$\max\limits_{t\in[0,\infty]}k(t)=k(t_{0})=I(t_{0}u,t_{0}v).$$ 
% From the above analysis, we have 
% \begin{equation*}
% k^{\prime}(t)>0  \text{ for } t<t_{0} \text{ and } k^{\prime}(t)<0  \text{ for } t>t_{0}. 
% \end{equation*}
Consequently, $k^{\prime}(t_{0})=0$ and hence $(t_0u,t_0v)\in\mathcal{N}$. In particular, if $(u,v)\in\mathcal{N},$ then $k^{\prime}(1)=0$. Therefore,  $$\max\limits_{t\in[0,\infty]}k(t)=k(1)=I(u,v).$$
For any $(u,v)\in\mathcal{N}$, consider the map $\gamma:[0,1]\rightarrow  \W$ defined as $\gamma(t)=(tt_{1}u,tt_{1}v),$ where $t_{1}>0$ such that $I(t_{1}u,t_{1}v)<0$. This implies that $\gamma\in\varGamma$. 
Therefore, we have $$c_1\leq\max\limits_{t\in[0,1]}I(\gamma(t))=\max\limits_{t\in[0,1]}I(tt_{1}u,tt_{1}v)\leq\max\limits_{t\geq 0}I(tu,tv)=I(u,v).$$ 
Hence,  
\begin{equation}\label{min_2}
    I(\tilde{u},\tilde{v})  =c_1\leq \inf_{(u,v)\in\mathcal{N}}I(u,v)=l_\mathcal{N}.
\end{equation}
 From \eqref{min_1} and \eqref{min_2}, we have 
 \begin{equation*}
l_\mathcal{N}= l_{S}=c_1=I(\tilde{u},\tilde{v}),
\end{equation*}
which proves that $(\tilde{u},\tilde{v})$ has the least energy.
\end{proof}

Next, we show that $(\tilde{u},\tilde{v})$ is non semi-trivial. Take $\tilde{\eta}\ge 0$. To show the dependency on $\tilde{\eta}$, we denote $I, \n, \tilde{u},\tilde{v}$ by $I_{\tilde{\eta}}, \n_{\tilde{\eta}}, u_{\tilde{\eta}},v_{\tilde{\eta}}$ respectively. In view of this,
$$\theta(\tilde{\eta}):=\inf_{\n_{\tilde{\eta}}} 
 I_{\tilde{\eta}}.$$
Recall that for every $(u,v)\in \W\setminus\{(0,0)\}$,
\begin{enumerate}[label=(\roman*)]
    \item\label{1} $\frac{d}{dt}I_{\tilde{\eta}}(tu,tv)=0$ if and only if  $(tu,tv)\in\n_{\tilde{\eta}}$.
    \medskip
    \item\label{2} there exists a unique $t_{u,v}\in(0,\infty)$ such that $I_{\tilde{\eta}}(t_{u,v}u,t_{u,v}v) = \max\{ I_{\tilde{\eta}}(tu,tv) : t \in  [0,\infty)\}$.
\end{enumerate}
We also recall $(u_0, v_0) \in \W$ defined in \eqref{def}. Then there exists $t_{u_0,v_0} \in (0, \infty)$ such that $$(t_{u_0,v_0} u_0, t_{u_0,v_0} v_0) \in \mathcal{N}_{\tilde{\eta}}.$$ In view of \eqref{en-1}, for $\tilde{\eta} = \eta>0$ we have
$$\theta(\eta) \leq I_\eta(t_{u_0,v_0}u_0,t_{u_0,v_0}v_0)= \sup_{t\geq0}I_\eta(tu_0,tv_0)<\tilde{c}.$$

\begin{remark}\label{lower-bound-1}
For $0\leq\eta_2\leq\eta_1$ and $(u,v)\in\n_{\eta_2}$, 
    \begin{equation*}
        \theta(\eta_1)\leq \max_{t\geq0}I_{\eta_1}(tu,tv)\leq  \max_{t\geq0}I_{\eta_2}(tu,tv)=I_{\eta_2}(u,v).
    \end{equation*}
    This implies that $\theta(\eta_2)\geq\theta(\eta_1)$. Hence, $\theta$ is non-increasing. Observing the proof of Lemma \ref{ps3}-(a), $\Theta$ is independent of $\eta$ for every $\eta\in(0,\frac{\la_1}{2})$. Thus, $\theta(\eta)\geq\Theta$ uniformly for all $\eta\in[0,\frac{\la_1}{2}).$ 
\end{remark}
\begin{lemma}\label{aux_lem_1-1.1}
    If $\eta\to0$, then $\theta(\eta)\to \tilde{c}$.
\end{lemma}
\begin{proof}
    First, we claim that $\theta(0)=\tilde{c}$. Observe that
    $$\max_{t\in[0,\infty)}I_0(tu_0,tv_0)=\left(\frac1p-\frac{1}{\psm}\right)\gamma^{-\frac{p}{p_s^*(m)-p}}\left(\frac{\|(u_0,v_0)\|_{\W}^p}{(\int_{\Om}\frac{|u_0|^\al|v_0|^\be}{|x|^m}\dx)^{\frac{p}{\psm}}}\right)^{\frac{\psm}{\psm-p}},$$
    where $u_0=\al^{\frac1p}u_{m,\var,1},v_0=\be^{\frac1p}u_{m,\var,1}$.
    By the definition of $\theta(0)$ and following the proof of Lemma \ref{ps2}, 
    $$\theta(0)\leq \max_{t\in[0,\infty)}I_0(tu_0,tv_0)\leq \tilde{c}+O(\var^{\frac{d-ps}{p-1}})+O(\var^{\frac{d-m}{p-1}}).$$
    Taking $\var\to0$, we get $\theta(0) \le \tilde{c}$. Let $(u_n,v_n)\in\W$ be a sequence such that $I_0(u_n,v_n)\to\theta(0)$ and $I_0'(u_n,v_n)\to0$. As a result
    \begin{align}\label{par8-1}
        &\frac1p\|(u_n,v_n)\|_{\W}^p-\frac{\gamma}{\psm}\int_{\Om}\frac{|u_n|^\al|v_n|^\be}{|x|^m}\dx=\theta(0)+o_n(1), \no\\
        & \|(u_n,v_n)\|_{\W}^p=\gamma\int_{\Om}\frac{|u_n|^\al|v_n|^\be}{|x|^m}\dx+o_n(1).
    \end{align}
    Using \eqref{par8-1}, 
    \begin{align}\label{par10-1}
        \left(\frac1p-\frac1{\psm}\right)\|(u_n,v_n)\|_{\W}^p=\theta(0)+o_n(1).
    \end{align}
    From Remark \ref{lower-bound-1} we get $\theta(0)>0$ and hence using \eqref{par10-1}, $u_n,v_n \neq 0$ for large $n$. Moreover, using the definition of $S_{\al,\be}$ and \eqref{par8-1},
    \begin{align}\label{par10.5-1}
    S_{\al,\be}\leq \frac{\|(u_n,v_n)\|_{\W}^p}{(\int_{\Om}\frac{|u_n|^\al|v_n|^\be}{|x|^m}\dx)^{\frac{p}{\psm}}}=\gamma^{\frac p{\psm}}\|(u_n,v_n)\|_{\W}^{p(1-\frac{p}{\psm})}+o_n(1).
    \end{align}
   Hence \eqref{par10-1} and \eqref{par10.5-1} yield $\theta(0)\geq \tilde{c}$ and the claim holds. Next, let $\{\eta_n\}$ be a sequence satisfying $\eta_n\downarrow0$. For every $n$, we have $\theta(\eta_n)<\tilde{c}$ and there exists $(u_{\eta_n},v_{\eta_n})\in \n_{\eta_n}$ (for brevity, we write $(u_n,v_n)$ instead of $(u_{\eta_n},v_{\eta_n})$) such that  $I_{\eta_n}(u_n,v_n)=\theta(\eta_n)$. Furthermore, there exists $t_n>0$ such that 
    \begin{align*}
        \theta(0)&\leq \max_{t\geq0}I_0(tu_n,tv_n)=I_0(t_nu_n,t_nv_n)=I_{\eta_n}(t_nu_n,t_nv_n)+\frac{\eta_n}{r} t_n^r\int_{\Om}(|u_n|^r+|v_n|^r)\dx\\
        &\leq \max_{t\geq0}I_{\eta_n}(tu_n,tv_n)+\frac{\eta_n}{r}  t_n^r\int_{\Om}(|u_n|^r+|v_n|^r)\dx\\
        &=I_{\eta_n}(u_n,v_n)+\frac{\eta_n}{r}t_n^r\int_{\Om}(|u_n|^r+|v_n|^r)\dx=\theta(\eta_n)+\frac{\eta_n}{r}t_n^r\int_{\Om}(|u_n|^r+|v_n|^r)\dx.
    \end{align*}
    Observe that the sequence $\{(u_n,v_n)\}$ is bounded in $\W$. Consequently, we get the convergence if $\{t_n\}$ is bounded. On the contrary, suppose $t_n\to\infty$. Since $(t_nu_n, t_n v_n)\in\n_{0}$, 
    \begin{equation}\label{para3-1}
        \|(u_n,v_n)\|_{\W}^p=t_n^{\psm-p}{\gamma}\int_{\Om}\frac{|u_n|^\al|v_n|^\be}{|x|^m}\dx.
    \end{equation}
    Since $\|(u_n,v_n)\|_{\W}$ is bounded, \eqref{para3-1} implies $$\int_{\Om}\frac{|u_n|^\al|v_n|^\be}{|x|^m}\dx=o_n(1).$$
    Since $(u_n,v_n)\in\n_{\eta_n}$, $$\|(u_n,v_n)\|_{\W}^p=\gamma\int_{\Om}\frac{|u_n|^\al|v_n|^\be}{|x|^m}\dx+\eta_n \int_{\Om}(|u_n|^{r}+|v_n|^{r})\dx.$$
Further, $\eta_n\int_{\Om}(|u_n|^r+|v_n|^r)\dx=o_n(1)$.
Thus, $(u_n,v_n)\to(0,0)$ in $\W$. Since $I_{\eta_n}(u_n,v_n)=\theta(\eta_n)$, we get $\theta(\eta_n)= o_n(1).$ This contradicts the fact that $\theta(\eta_n)>\Theta$ for every $n$ large enough (see Remark \ref{lower-bound-1}). This completes the proof.
\end{proof}

\begin{proposition}\label{submain-2}
Let $0 \leq m \le sp$ and $p \le r < p^*_s$. Let $\al,\be>1$ be such that $\al + \be = p^*_s(m)$. Let  $\gamma, r$ be as given in Proposition \ref{subexistence&uniqueness}. Then there exists $0<\eta^* < \min\{\lambda_1, \frac{S_{\al+\be}}{c_{0}}\}$ such that for every $\eta \in (0, \eta^*)$, $(u_{\eta}, v_{\eta}) \neq (u_{\eta},0)$ or $(u_{\eta}, v_{\eta}) \neq (0,v_{\eta})$ a.e. in $\Omega$.
\end{proposition}

\begin{proof}
   Using $\theta(\eta)=I_{\eta}(u_{\eta},v_{\eta})$ and $(u_{\eta},v_{\eta})\in\n_{\eta}$, we get
    $$\theta(\eta)=\left(\frac1p-\frac1r\right)\|(u_{\eta},v_{\eta})\|_{\W}^p-\left(\frac1{\psm}-\frac1r\right)\gamma\int_{\Om}\frac{|u_{\eta}|^\al|v_{\eta}|^\be}{|x|^m}\dx.$$
    Since $\theta(\eta)<\tilde{c}$, $\{(u_{\eta},v_{\eta})\}_{\eta}$ is bounded in $\W$. Again since $(u_{\eta},v_{\eta})\in\n_{\eta}$, we have $$\|(u_{\eta},v_{\eta})\|_{\W}^p=\gamma\int_{\Om}\frac{|u_{\eta}|^\al|v_{\eta}|^\be}{|x|^m}\dx+o_\eta(1).$$ Using Lemma \ref{aux_lem_1-1.1}, $\theta(\eta)\to\tilde{c}$ as $\eta\to0$. Thus, 
    $$0<\tilde{c}=\left(\frac1p-\frac1{\psm}\right)\gamma\int_{\Om}\frac{|u_{\eta}|^\al|v_{\eta}|^\be}{|x|^m}\dx+o_\eta(1).$$
Hence, $(u_{\eta}, v_{\eta}) \neq (u_{\eta},0)$ or $(u_{\eta}, v_{\eta}) \neq (0,v_{\eta})$ a.e. in $\Omega$.   
\end{proof}

\noi \textbf{Proof of Theorem \ref{existence-main}:} Proof follows using Proposition \ref{subexistence&uniqueness} and Proposition \ref{submain-2}. \qed

\begin{remark}[Positive solution]\label{pos}
 For $(u,v)\in\W$, we define
\begin{equation}\label{I+}
I_+(u,v)=\frac{1}{p}\|(u,v)\|^p_{\mathcal{W}}
-\frac{\gamma}{p^*_s(m)}\int_{\Omega} \frac{(u^+)^{\alpha}(v^+)^{\beta}}{\abs{x}^m} \dx - \frac{\eta}{r} \int_{\Omega} \left( (u^+)^r + (v^+)^r \right) \dx.
\end{equation} Observe that $I_+\in C^1(\W,\R)$ and every critical point of $I_+$ corresponds to a non-negative weak solution of \eqref{system_eqn}. We remark that Lemmas \ref{ps1}--\ref{ps3} hold for $I_+$ as well. As a consequence, we get a non-negative, non semi-trivial solution $(\hat{u},\hat{v})$ of \eqref{system_eqn}, provided $\gamma,\eta$ and $r$ as in Theorem \ref{existence-main}. Further, taking $(\hat{u},0)$ and $(0,\hat{v})$ as test functions in $\W$ and using the fact that $(\hat{u},\hat{v})$ weakly solves \eqref{system_eqn}, we see that the following hold weakly: 
\begin{align*}
    (-\Delta_p)^s \hat{u} \ge 0, \text{ and } (-\Delta_p)^s \hat{v} \ge 0 \text{ a.e. in } \Omega.
\end{align*}
Moreover, since $\hat{u}, \hat{v} \neq 0$, we apply the strong maximum principle (see \cite[Lemma A.1]{BrFr}) to get $\hat{u}, \hat{v} >0$ a.e. in $\Omega$.
\end{remark}

\section{Existence of category many nontrivial solutions}

In this section, we assume that $m=0, p, \al, \be \ge 2, \al+\be=p^*_s$ and consider the following condition on $\eta,r$:
\begin{align*}
     \eta < \min\left\{\frac{\la_1}{2},\frac{S_{\al+\be}}{c_0}\right\} \text{ when } r=p, \text{ and } \begin{cases}
r>p^{*}_s-p^{\prime}, & \text { if } d<p^{2} s;\\
r \geq p, & \text { if } d \geq p^2s,
\end{cases}
\end{align*}
where $c_0$ is given in \eqref{c0}.
We define $\psi_\eta:\W\to\R$ by $$\psi_\eta(u,v)=\langle I_\eta'(u,v),(u,v)\rangle=\|(u,v)\|_{\W}^p-\gamma\int_{\Omega} {\left|u\right|^{\alpha}\left|v\right|^{\beta}}\dx-\eta \int_{\Omega}\left(\left|u\right|^{r}+\left|v\right|^{r}\right)\dx.$$
Observe that $\psi_\eta$ is of class $\C^2$ and
$$\langle\psi_\eta'(u,v),(u,v)\rangle=p\|(u,v)\|_{\W}^p-\gamma\, p_s^*\int_{\Omega} {\left|u\right|^{\alpha}\left|v\right|^{\beta}}\dx-\eta r \int_{\Omega}\left(\left|u\right|^{r}+\left|v\right|^{r}\right)\dx,$$
and for every $(u,v)\in\n_\eta$ (defined in \eqref{nehari}),
\begin{equation}\label{aaa}
    \langle\psi_\eta'(u,v),(u,v)\rangle=\gamma(p- p_s^*)\int_{\Omega} {\left|u\right|^{\alpha}\left|v\right|^{\beta}}\dx+\eta(p- r) \int_{\Omega}\left(\left|u\right|^{r}+\left|v\right|^{r}\right)\dx<0.
\end{equation}
Thus, $I_\eta,\psi_\eta,\n_\eta$ satisfy the structural assumptions given in \cite[Section 5.3]{willem}.

\begin{definition}
 We say that $I_\eta$ restricted on $\n_\eta$ satisfies PS condition at level $c$, $\text{(PS)}_c^{\n_\eta}$ in short, if for $\{(u_n,v_n)\}\subset\n_\eta$ such that
\begin{equation}\label{ps_rest}
I_\eta(u_n,v_n)\to c,\;\min_{\la\in\R}\|I_\eta'(u_n,v_n)-\la\psi_\eta'(u_n,v_n)\|_{\W^*}\to0,\; \text{ as }n\to\infty,  
\end{equation}
has a convergent sequence in $\W$.   
\end{definition}

\begin{lemma}\label{ps4}
$I_\eta$ satisfies (PS$)_c^{\n_\eta}$ for $c<\tilde{c}$.
\end{lemma}
\begin{proof}
   Suppose $\{(u_n,v_n)\}\subset\n_\eta$ satisfies \eqref{ps_rest} with $c<\tilde{c}$. As in the proof of Lemma \ref{ps1}, using $\langle I_\eta'(u_n,v_n),(u_n,v_n)\rangle=0$ we see that $\{(u_n,v_n)\}$ is a bounded sequence in $\W$. Thus, the sequence $\{\langle\psi_\eta'(u_n,v_n),(u_n,v_n)\rangle\}_{n}\subset\R$ is bounded. As a result, by \eqref{aaa} there exist $l\in(-\infty,0]$ such that $$\langle\psi_\eta'(u_n,v_n),(u_n,v_n)\rangle\to l, \; \text{ as } n\to\infty.$$ 
   By \eqref{ps_rest}, there exists $\la_n\in\R$ such that $$\|I_\eta'(u_n,v_n)-\la_n\psi_\eta'(u_n,v_n)\|_{\W^*}\to0$$
   as $n\to\infty$. Suppose $l=0$. By \eqref{aaa}, 
   $$\int_{\Omega} {\left|u_n\right|^{\alpha}\left|v_n\right|^{\beta}}\dx\to0,\; \int_{\Omega}\left(\left|u_n\right|^{r}+\left|v_n\right|^{r}\right)\dx\to0.$$
   Since, $(u_n,v_n)\in\n_{\eta}$, we immediately get $\|(u_n,v_n)\|_{\W}\to0$. Thus, $(u_n,v_n)\to0$ in $\W$. Suppose $l<0$. Since $\|I_\eta'(u_n,v_n)-\la_n\psi_\eta'(u_n,v_n)\|_{\W^*}\to0$ as $n\to\infty$, $$\langle I_\eta'(u_n,v_n),(u_n,v_n)\rangle-\la_n\langle\psi_\eta'(u_n,v_n),(u_n,v_n)\rangle\to0.$$
   Further, using $(u_n,v_n)\in\n_\eta$, we get $\la_n\to0.$ Thus $\{(u_n,v_n)\}$ is a PS sequence for $I_\eta$ in $\W$ at the level $c$. Now, using Lemma \ref{ps1}, we conclude the proof.
\end{proof}

\begin{remark}\label{remark_cat}
Observe that $I_{\eta}$ is bounded below on $\n_\eta$ and by the above lemma $I_{\eta}$ satisfies (PS$)_c^{\n_\eta}$ for $c<\tilde{c}$. Therefore, applying \cite[Theorem 5.20]{willem}, we get at least $\text{cat}_{I_\eta^k}(I_\eta^k)$ critical points of $I_\eta$ restricted to $\n_\eta$ for $k\in[\theta(\eta),\tilde{c})$, where the sublevel set $I_\eta^k$ is defined as $$I_\eta^{k}:=\left\{ (u,v)\in\n_\eta:I_\eta(u,v)\leq k \right\}.$$
\end{remark}
We set $B_\tau:=B_\tau(0)\subset\Om$, define $\W^{\text{rad}}=\{(u,v)\in\W(B_\tau):u,v\text{ are radial}\}$ and the functional $$I_{\eta,\tau}(u,v)=\frac{1}{p}\|(u,v)\|^p_{\mathcal{W}}-\frac{\gamma}{p_s^*}\int_{B_\tau}{|u|^\alpha|v|^\be}\dx-\frac{\eta}{r}\int_{B_\tau}(|u|^r+|v|^r)\dx,\; \forall \, (u,v) \in \W^{\text{rad}}.$$ Then we consider the following set
\begin{align*}
    \n_{\eta,\tau}^{\text{rad}}=\left\{(u,v)\in \W^{\text{rad}}:\langle I_{\eta,\tau}'(u,v),(u,v)\rangle=0\right\}.
\end{align*}
%By Theorem \ref{existence-main}, for $\Om=B_\tau$, \eqref{system_eqn} admits a solution  $(u_{\eta,\tau},v_{\eta,\tau})$ with the least energy. 
The discussion after Proposition \ref{subexistence&uniqueness}, particularly  \ref{1} and \ref{2}, also hold for $I_{\eta,\tau}$. Since $u_{0,\varepsilon,1}$ is radial (see Remark \ref{radial-function}) and for $(u_0,v_0)$ defined as in \eqref{def}, since $(t_{u_0,v_0}u_0,t_{u_0,v_0}v_0)\in\n_{\eta}$, we infer that $\n_{\eta,\tau}^{\text{rad}}$ is nonempty and $$\theta(\eta)\leq\theta_1(\eta):=\inf_{\n_{\eta,\tau}^{\text{rad}}}I_{\eta,\t}<\tilde{c}.$$
Applying Lemma \ref{ps4} and the Ekeland variational principle, there exists $(u_{\eta,\tau},v_{\eta,\tau})\in\n_{\eta,\tau}^{\text{rad}}$ such that $I_{\eta,\tau}(u_{\eta,\tau},v_{\eta,\tau})=\theta_1(\eta)$. Now our aim is to show that for $\eta$ small enough, $\text{cat}_{I_\eta^{\theta_1(\eta)}}(I_\eta^{\theta_1(\eta)})$ is bigger than $\text{cat}_{\Om}(\Om).$
For $\tau>0$, we define $$\Om_\tau^+:=\left\{x\in\R^d:d(x,\Om)<\tau\right\},\;\Om_\tau^-:=\left\{x\in\Om:d(x,\pa\Om)>\tau\right\}.$$
Clearly, $\Om_\tau^+$ and $\Om_\tau^-$ are homotopically equivalent to $\Om$. Now, we define the barycenter map $\Phi_\eta:\n_\eta\to\R^d$ by
\begin{equation*}
    \Phi_\eta(u,v)=\frac{\gamma}{\gamma^{-\frac{p}{p_s^*-p}}S_{\al,\be}^{\frac{d}{sp}}}\int_{\Om}x|u|^\al|v|^\be\dx,\quad\forall (u,v)\in\n_\eta.
\end{equation*}
Observe that $\Phi_\eta$ is continuous.
\begin{lemma}\label{aux_lem_3}
    There exists $\eta_{\ast\ast}>0$ such that for $\eta\in(0,\eta_{\ast\ast})$, $\Phi_{\eta}(u,v)\in\Om_\tau^+$ for every $(u,v)\in I_{\eta}^{\theta_1(\eta)}$.
\end{lemma}
\begin{proof}
    Arguing by contradiction, suppose there exists $\eta_n>0$, $\eta_n\to0$ and $(u_n,v_n)\in I_{\eta_n}^{\theta_1(\eta_n)}$ such that $\Phi_{\eta_n}(u_n,v_n)\not\in\Om_\tau^+$. Then
    \begin{align}
        &\|(u_n,v_n)\|_{\W}^p-\gamma\int_{\Omega} {\left|u_n\right|^{\alpha}\left|v_n\right|^{\beta}}\dx-\eta_n \int_{\Omega}\left(\left|u_n\right|^{r}+\left|v_n\right|^{r}\right)\dx=0,\label{par1}\\
        &\frac{1}{p}\|(u_n,v_n)\|_{\W}^p-\frac{\gamma}{p_s^*}\int_{\Om}{|u_n|^\alpha|v_n|^\be}\dx-\frac{\eta_n}{r}\int_{\Om}(|u_n|^r+|v_n|^r)\dx\leq {\theta_1(\eta_n)}<\tilde{c}.\label{par2}
    \end{align}
    Thus $I_{\eta_n}(u_n,v_n)-I_{\eta_n}'(u_n,v_n)(u_n,v_n)< \tilde{c}$. Hence arguing as in Lemma \ref{ps1}, $\{(u_n,v_n)\}$ is bounded in $\W$. Thus the sequences $\{\|u_n\|_r\},\{\|v_n\|_r\}$ are bounded. As a consequence, $\eta_n\int_{\Om}(|u_n|^r+|v_n|^r)\dx = o_n(1)$. So we can rewrite \eqref{par1} and \eqref{par2} as
    \begin{align}
        &\|(u_n,v_n)\|_{\W}^p-\gamma\int_{\Omega} {\left|u_n\right|^{\alpha}\left|v_n\right|^{\beta}}\dx=o_n(1),\label{par-3}\\
       \notag &\frac{1}{p}\|(u_n,v_n)\|_{\W}^p-\frac{\gamma}{p_s^*}\int_{\Om}{|u_n|^\alpha|v_n|^\be}\dx\leq {\theta_1(\eta_n)}+o_n(1).
    \end{align}
    Thus, 
    \begin{equation*}
        \left(\frac1p-\frac1{p_s^*}\right)\|(u_n,v_n)\|_{\W}^p\leq {\theta_1(\eta_n)}+o_n(1)\leq \tilde{c}+o_n(1).
    \end{equation*}
    By the definition of $\tilde{c}$, 
     \begin{equation}\label{par5}
        \|(u_n,v_n)\|_{\W}^p\leq \gamma^{-\frac{p}{p_s^*-p}}S_{\al,\be}^{\frac{d}{sp}}+o_n(1).
    \end{equation}
    Since $\Phi_{\eta_n}(u_n,v_n)\not\in\Om_\tau^+$, $u_n,v_n\not=0$. Further, using the definition of $S_{\al,\be}$ and \eqref{par-3},
    \begin{equation}\label{par6}
        S_{\al,\be}\leq \frac{\|(u_n,v_n)\|_{\W}^p}{(\int_{\Om}{|u_n|^\alpha|v_n|^\be}\dx)^{\frac{p}{p_s^\ast}}}=\gamma^{\frac{p}{p_s^\ast}}\|(u_n,v_n)\|_{\W}^{p(1-\frac{p}{p_s^\ast})}+o_n(1).
    \end{equation}
    Combining \eqref{par-3}, \eqref{par5} and \eqref{par6}, we have
    \begin{equation}\label{par7}
        \|(u_n,v_n)\|_{\W}^p=\gamma\int_{\Om}{|u_n|^\alpha|v_n|^\be}\dx+o_n(1)=\gamma^{-\frac{p}{p_s^*-p}}S_{\al,\be}^{\frac{d}{sp}}+o_n(1).
    \end{equation}
    Letting $$\ov{u}_n=\frac{u_n}{\left(\int_{\Om}|u_n|^\alpha|v_n|^\be\dx\right)^{\frac{1}{p_s^\ast}}},\quad \ov{v}_n=\frac{v_n}{\left(\int_{\Om}|u_n|^\alpha|v_n|^\be\dx\right)^{\frac{1}{p_s^\ast}}}.$$
    By the definition of $S_{\alpha,\beta}$, \eqref{par5} and \eqref{par6},
    \begin{equation*}
        S_{\al,\be}\leq \|(\ov{u}_n,\ov{v}_n)\|_{\W}^p=\frac{\|(u_n,v_n)\|_{\W}^p}{\left(\int_{\Om}|u_n|^\alpha|v_n|^\be\dx\right)^{\frac{p}{p_s^\ast}}}=\gamma^{\frac p{p_s^*}}\|(u_n,v_n)\|_{\W}^{p(1-\frac p{p_s^*})}\leq S_{\al,\be}+o_n(1).
    \end{equation*}
    Hence, as $n\to\infty$, the sequence $\{(\ov{u}_n,\ov{v}_n)\}$ satisfies 
    $$\int_{\Om}|\ov{u}_n|^\al|\ov{v}_n|^\be\dx=1,\quad \|(\ov{u}_n,\ov{v}_n)\|_{\W}^p\to S_{\al,\be}.$$
    Thus, by Proposition \ref{par_the_rescale}, there exists $\{(y_n,r_n)\}\subset\R^d\times\R^+$ such that the sequence $$(\tilde{u}_n,\tilde{v}_n)=r_n^{\frac{d-sp}{p}}\left(\ov{u}_n(r_nx+y_n),\ov{v}_n(r_nx+y_n) \right)$$ has a convergent subsequence, still denoted by $(\tilde{u}_n,\tilde{v}_n)$, such that $(\tilde{u}_n,\tilde{v}_n)\to (\tilde{u},\tilde{v})$ in $\W(\rd)$. Also, observe that 
    \begin{align}\label{cv-1}
        1 = \int_{\Om}|\ov{u}_n|^\al|\ov{v}_n|^\be\dx = \int_{\rd} |\ov{u}_n|^\al|\ov{v}_n|^\be\dx = \int_{\rd} |\tilde{u}_n|^\al|\tilde{v}_n|^\be\dx = \int_{\rd} |\tilde{u}|^\al|\tilde{v}|^\be\dx + o_n(1). 
    \end{align}
    Moreover, $r_n\to0$ and $y_n\to y$ in $\ov{\Om}$. Let $\varphi\in \C_c^{\infty}(\R^d)$ be such that $\varphi(x)=x$ on $\ov{\Om}$. Then, by \eqref{par7} and applying the dominated convergence theorem, we obtain 
\begin{align*}
    \Phi_{\eta_n}(u_n,v_n)&=\frac{\gamma}{\gamma^{-\frac{p}{p_s^*-p}}S_{\alpha,\be}^{\frac d{sp}}}\int_{\Om}x|u_n|^{\al}|v_n|^{\be}\dx=\frac{\gamma\int_{\Om}|u_n|^{\al}|v_n|^{\be}\dx}{\gamma^{-\frac{p}{p_s^*-p}}S_{\alpha,\be}^{\frac d{sp}}}\int_{\Om}x|\ov{u}_n|^{\al}|\ov{v}_n|^{\be}\dx\\
    &= (1+o_n(1))\int_{\R^d}\varphi(x)|\ov{u}_n|^{\al}|\ov{v}_n|^{\be}\dx=\int_{\R^d}\varphi(r_nz+y_n)|\tilde{u}_n(z)|^{\al}|\tilde{v}_n(z)|^{\be}\dz+o_n(1)\\
    &=y\int_{\R^d}|\tilde{u}(z)|^{\al}|\tilde{v}(z)|^{\be}\dz=y+o_n(1),
\end{align*}
    where the last identity follows using \eqref{cv-1}. 
    Thus, $\Phi_{\eta_n}(u_n,v_n)\to y\in\ov{\Om}$, which contradicts the fact that $\Phi_{\eta_n}(u_n,v_n)\not\in\Om_\tau^+$.
\end{proof}

 Taking $\Omega=B_{\tau}$ in Lemma \ref{aux_lem_1-1.1}, we observe that 
\begin{equation}\label{eq_1}
   \lim_{\eta \to 0} \theta_1(\eta)=\tilde{c}.
\end{equation} 
Using the same arguments as in Proposition \ref{submain-2} and  \eqref{eq_1} we get the following lemma.
\begin{lemma}\label{aux_lem_2}
    It holds $$ \lim_{\eta \to 0} \frac{\gamma\int_{B_{\tau}(0)}|u_{\eta,\tau}|^\al|v_{\eta,\tau}|^\be\dx}{\gamma^{-\frac{p}{p_s^*-p}}S_{\al,\be}^{\frac{d}{sp}}}=1.$$
\end{lemma}
Now, we define the map $\zeta_\tau:\Om_\tau^-\to I_\eta^{\theta_1(\eta)}$ by \begin{equation}\label{eta_fn}
    \zeta_\tau(y)(x):=\begin{cases}
        (u_{\eta,\tau}(x-y),v_{\eta,\tau}(x-y)),&x\in B_\tau(y);\\(0,0),&x\not\in B_\tau(y),
    \end{cases}
\end{equation}
where $y\in\Om_\tau^-$. Observe that $\zeta_\tau$ is well-defined and continuous. Further, for $y\in\Om_\tau^-$,
\begin{align*}
    \Phi_\eta(\zeta_\t(y))&=\frac{\gamma}{\gamma^{-\frac{p}{p_s^*-p}}S_{\al,\be}^{\frac{d}{sp}}}\int_{B_\tau(y)}x|u_{\eta,\tau}(x-y)|^\al|v_{\eta,\tau}(x-y)|^\be\dx\\&=\frac{\gamma}{\gamma^{-\frac{p}{p_s^*-p}}S_{\al,\be}^{\frac{d}{sp}}}\int_{B_\tau(0)}(y+z)|u_{\eta,\tau}(z)|^\al|v_{\eta,\tau}(z)|^\be\dz\\
    &=\frac {\gamma\,y}{\gamma^{-\frac{p}{p_s^*-p}}S_{\al,\be}^{\frac{d}{sp}}}\int_{B_\tau(0)}|u_{\eta,\tau}(z)|^\al|v_{\eta,\tau}(z)|^\be\dz,
\end{align*}
where in the last line, we use the fact that $u_{\eta,\tau},v_{\eta,\tau}$ are radial and thus $$\int_{B_\tau(0)}z|u_{\eta,\tau}(z)|^\al|v_{\eta,\tau}(z)|^\be\dz=0.$$
To see this, we calculate for each coordinate $z_i$ of $z$,
\begin{align*}
    \int_{B_\tau(0)}z_i|u_{\eta,\tau}(z)|^\al|v_{\eta,\tau}(z)|^\be\dz &= \int_{0}^\tau \int_{\pa B_{\tau}} z_i |u_{\eta,\tau}(r)|^\al|v_{\eta,\tau}(r)|^\be  \d S \d r\\&=\int_{0}^\tau |u_{\eta,\tau}(r)|^\al|v_{\eta,\tau}(r)|^\be\left(\int_{\pa B_{\tau}} z_i\d S\right)\d r. 
\end{align*}
By either using spherical coordinates or the Divergence theorem, we can see that $\int_{\pa B_{\tau}}z_i\d S=0.$
Next, we define the map $H_\eta:[0,1]\times I_\eta^{\theta_1(\eta)}\to \R^d$ by $$H_\eta(t,(u,v))=\left(t+(1-t) \frac{\gamma^{-\frac{p}{p_s^*-p}}S_{\al,\be}^{\frac{d}{sp}}}{\gamma\int_{B_{\tau}}|u_{\eta,\tau}|^\al|v_{\eta,\tau}|^\be\dx}\right)\Phi_\eta(u,v).$$
By Lemma \ref{aux_lem_2}, we observe that for $\eta$ small enough, $u_{\eta,\tau}\neq0, v_{\eta,\tau}\neq0$  a.e. in $B_{\tau}$. Thus, $H_\eta$ is well-defined and continuous. 
\begin{lemma}\label{aux_lem_4}
    There exists $\eta_{\ast\ast\ast}>0$ such that for $\eta\in(0,\eta_{\ast\ast\ast})$, $H_\eta\in \C([0,1]\times I^{\theta_1(\eta)},\Om_\tau^+).$
\end{lemma}
\begin{proof}
   We first show that $H_{\eta}(t, (u,v)) \in \Omega_{\tau}^+$ for $(u,v) \in I^{\theta_1(\eta)}$ and $t \in [0,1]$. On the contrary, suppose there exists $\eta_n\downarrow0$, $t_n\to t_0\in[0,1]$ and $(u_n,v_n)\in I^{\theta_1(\eta_n)}$ but $H_{\eta_n}(t_n,(u_n,v_n))\not\in\Om_\tau^+$. By Lemma \ref{aux_lem_2} and Lemma \ref{aux_lem_3}, we have $$\frac{\gamma\int_{B_{\tau}}|u_{\eta_n,\tau}|^\al|v_{\eta_n,\tau}|^\be\dx}{\gamma^{-\frac{p}{p_s^*-p}}S_{\al,\be}^{\frac{d}{sp}}}\to1, \text{ and } \Phi_{\eta_n}(u_n,v_n)\to y\in\ov{\Om},$$
    and hence $H_{\eta_n}(t_n,(u_n,v_n))\to y\in\ov{\Om},$
    which is a contradiction. The continuity of $H_\eta$ is immediate from it's definition.
\end{proof}

\begin{lemma}\label{lemma-cat}
   There exists $\eta_\ast> 0$ such that for $\eta\in(0,\eta_\ast)$, 
    \begin{equation*}
\text{cat}_{I_\eta^{\theta_1(\eta)}}(I_\eta^{\theta_1(\eta)})\geq\text{cat}_{\Om}(\Om).
    \end{equation*}
\end{lemma}
\begin{proof}
Set $\eta_\ast:=\min\{\eta_{**},\eta_{***}\}$, where $\eta_{**}$ and $\eta_{***}$ are as in Lemma \ref{aux_lem_3} and Lemma \ref{aux_lem_4} respectively. Suppose $\text{cat}_{I_\eta^{\theta_1(\eta)}}(I_\eta^{\theta_1(\eta)})=N$. Then  
    $$I_\eta^{\theta_1(\eta)}=F_1\cup \cdots\cup F_N,$$
    where each $F_i$ is a closed and contractible set in $I_\eta^{\theta_1(\eta)}$. This means for each $1\leq i\leq N$, there exist $h_i\in \C([0,1]\times F_i:I_\eta^{\theta_1(\eta)})$ and tuples $(u_i,v_i)$ such that $$h_i(0,(u,v))=(u,v),\; h_i(1,(u,v))=(u_i,v_i),\;\forall\,(u,v)\in F_i.$$
    We define $G_i=\zeta_\tau^{-1}(F_i)\subset\Om_\tau^-$ for $1\leq i\leq N$, where $\zeta_\tau$ is given as in \eqref{eta_fn}. Since $\zeta_\tau$ is continuous, each $G_i$ is closed in $\Om_\tau^-$ and $$\Om_\tau^-=G_1\cup\cdots \cup G_N.$$
    We claim that each $G_i$ is contractible in $\Om_\tau^+$. To see this, we consider the map $g_i:[0,1]\times G_i\to\Om_\tau^+$ defined as $$g_i(t,y):=H_\eta(t,h_i(t,\zeta_\tau(y))).$$
Clearly each $g_i$ is continuous and $$g_i(0,y)=H_\eta(0,h_i(0,\zeta_\tau(y)))=H_\eta(0,\zeta_\tau(y))=\frac{\gamma^{-\frac{p}{p_s^*-p}}S_{\al,\be}^{\frac{d}{sp}}}{\gamma\int_{B_{\tau}}|u_{\eta,\tau}|^\al|v_{\eta,\tau}|^\be\dx}\Phi_{\eta}(\zeta_\t(y)))=y,$$
and $$g_i(1,y)=H_\eta(1,h_i(1,\zeta_\tau(y)))=H_\eta(1,(u_i,v_i))=\Phi_{\eta}(u_i,v_i).$$
Hence the claim holds. As a consequence, $$\text{cat}_{I_\eta^{\theta_1(\eta)}}(I_\eta^{\theta_1(\eta)})\geq\text{cat}_{\Om_\t^+}(\Om_\t^-)=\text{cat}_\Om(\Om),$$ as required.
\end{proof}
\noi \textbf{Proof of Theorem \ref{cat_theorem}:}
    From Remark \ref{remark_cat} and Lemma \ref{lemma-cat}, we get at least $\text{cat}_\Om(\Om)$ critical points of $I_\eta$ restricted on $\n_\eta$. Thus there exists $\la\in\R$ such that $I_\eta'(u,v)=\la\psi_\eta'(u,v)$ for $(u,v)\in\n$. In particular, $$0=\langle I_\eta'(u,v),(u,v)\rangle=\la\langle\psi_\eta'(u,v),(u,v)\rangle.$$
    Since $\langle I_\eta'(u,v),(u,v)\rangle=0$ and $\psi_\eta'(u,v),(u,v)\rangle<0$ for every $(u,v)\in\n$, we get $\la=0$. Therefore, the critical points of $I_\eta$ restricted on $\n_\eta$ are nontrivial weak solutions of \eqref{system_eqn}.
\qed

\section*{Declaration} 
\noindent \textbf{Funding.}	The research of N.~Biswas is partially supported by the SERB National Postdoctoral Fellowship (PDF/2023/000038). The research of P.~Das is partially supported by the NBHM Fellowship (0203/5(38)/2024-R\&D-II/11224).\\
\noindent \textbf{Competing interests.} The authors have no competing interests to declare that are relevant to the content of this article.\\
\noindent\textbf{Data availability statement.} Data sharing is not applicable to this article as no data sets were generated or analysed during the current study.

\bibliographystyle{abbrvnat}

\end{document}